\documentclass{amsart}

\usepackage{amsmath,amsfonts,amssymb}
\usepackage{amscd} 
\usepackage{array} 

\theoremstyle{plain}
\newtheorem{theorem}{Theorem}[section]
\newtheorem{lemma}[theorem]{Lemma}
\newtheorem{corollary}[theorem]{Corollary}
\newtheorem{proposition}[theorem]{Proposition}

\theoremstyle{definition}

\newtheorem{definition}[theorem]{Definition}
\newtheorem{example}[theorem]{Example}

\newtheorem{remark}[theorem]{Remark}

\newtheorem{Constructions}[theorem]{Constructions}

\newcommand\myref[1]{\ref{#1}}
\newcommand\myhead[1]{\smallskip\noindent\textbf{#1.}}

\DeclareMathOperator\ad{ad}

\DeclareMathOperator\Aut{Aut}

\DeclareMathOperator{\Cd}{C}

\DeclareMathOperator\crk{crk}

\DeclareMathOperator\diag{diag}

\DeclareMathOperator\End{End}

\DeclareMathOperator\Hom{Hom}

\DeclareMathOperator\Lp{L} 
\DeclareMathOperator\Mat{M}
\DeclareMathOperator\orth{\mathfrak{o}}

\DeclareMathOperator\rank{rank}

\DeclareMathOperator\rk{rk}

\DeclareMathOperator\spl{\mathfrak{sl}}
\DeclareMathOperator\ssp{\mathfrak{ssp}}
\DeclareMathOperator\spu{\mathfrak{su}}
\DeclareMathOperator\supp{supp}
\DeclareMathOperator\trace{tr}
\DeclareMathOperator\spann{span}

\newcommand \andd{\quad\text{and}\quad}

\newcommand \id {\text{id}}

\newcommand \order[1]{\left\vert #1 \right\vert}
\newcommand\ot {\otimes}
\newcommand \set[1]{\{#1 \}}
\newcommand \suchthat { \mid }
\newcommand \zero {\{0\}}

\newcommand \bbC {\mathbb C}

\newcommand \bbN {\mathbb N}
\newcommand \bbQ{\mathbb Q}

\newcommand \bbZ {\mathbb Z}

\newcommand \cA {\mathcal A}
\newcommand \cB {\mathcal B}
\newcommand \cC {\mathcal C}

\newcommand \cI {\mathcal I}
\newcommand \cJ {\mathcal J}

\newcommand \cL {\mathcal L}
\newcommand \cM {\mathcal M}

\newcommand \cQ {\mathcal Q}
\newcommand \cS {\mathcal S}

\newcommand \cU {\mathcal U}
\newcommand \cX {\mathcal X}
\newcommand \cZ {\mathcal Z}

\newcommand \fg {{\mathfrak g}}
\newcommand \fh {{\mathfrak h}}
\newcommand \ft {{\mathfrak t}}

\newcommand \al {\alpha}

\newcommand \ep {\varepsilon}
\newcommand \De {\Delta}
\newcommand \gm {\gamma}
\newcommand \Gm{\Gamma}
\newcommand \lm {\lambda}
\newcommand \Lm {\Lambda}
\newcommand \ph{\varphi}
\newcommand \sg {\sigma}

\newcommand\boldm    {{\boldsymbol m}}

\newcommand\boldsg   {{\boldsymbol{\sigma}}}

\newcommand \tal {{\tilde\alpha}}
\newcommand \tC {{\widetilde{C}}}
\newcommand \tDe {{\widetilde{\Delta}}}
\newcommand \tfh {{\widetilde {\mathfrak h}}}
\newcommand \tildeh {{\tilde h}}

\newcommand \tL {{\widetilde{\cL}}}

\newcommand \tQ {{\widetilde{Q}}}

\newcommand \tx {{\tilde x}}
\newcommand \ty {{\tilde y}}
\newcommand \tZ {{\tilde {\mathcal Z}}}
\newcommand \tZp {{\widetilde {\mathcal Z'}}}

\newcommand \tCp {{\widetilde{C'}}}
\newcommand \tDep {{\widetilde{\Delta'}}}
\newcommand \tfhp {{\widetilde {\mathfrak h'}}}
\newcommand \tLp {{\widetilde{\cL'}}}
\newcommand \tQp {{\widetilde{Q'}}}

\newcommand \cLp{{\cL'}}
\newcommand \Lmp{{\Lm'}}

\DeclareMathOperator\rkv{rkv}
\newcommand\base{{\Bbbk}}

\newcommand\Dec{\De^\times}
\newcommand\Dei{\De_{\text{ind}}}
\newcommand\Deic{\Dei^\times}
\newcommand{\ditto}{"}
\newcommand\dL[2]{\cL_{#1}^{#2}}
\newcommand\extra{\text{ex}}

\newcommand\lng{\text{lg}}
\newcommand\pab{{\langle \beta, \alpha^\vee \rangle}}
\newcommand\pr{{\ph_\text{r}}}
\newcommand\pe{{\ph_\text{e}}}
\newcommand\ps{{\ph_\text{s}}}
\newcommand\qg{x} 
\newcommand\short{\text{sh}}
\newcommand \Tindex[4]{{\vphantom{#2}}^{#1}{\textrm{#2}}_{#3}^{#4}}
\newcommand \type[2]{\textrm{#1}_{#2}}

\newcommand\nl{\newline}  
\newcommand\rtj[1]{\hphantom{}\hfill #1} 
\newcommand\drop{\hspace{0pt}\nl}  
\newcommand\relscale[1]{}

\begin{document}

\title[Lie tori]{Some isomorphism invariants for Lie tori}
\author{Bruce Allison}

\address[Bruce Allison]{Department of Mathematical and Statistical Sciences  \\ University of Alberta \\Edmonton, Alberta  T6G 2G1\\  Canada}
\email{ballison@ualberta.ca}
\date{June 11, 2011}

\begin{abstract}
In this paper we study the isomorphism problem for
centreless Lie tori that are fgc (finitely generated as modules over their centroid).
These Lie tori play a important role in the theory of extended affine Lie algebras
and of multiloop Lie algebras.  We introduce four isomorphism invariants for fgc
centreless Lie tori,  and use them together with known structural results
to investigate the classification problem  for fgc centreless Lie tori up to isomorphism.  
\end{abstract}

\maketitle

Suppose that  $\base$ is a field of characteristic 0,
$\Lm$  is a finitely generated free abelian group, and $\De$
is an  irreducible finite root system.
A Lie torus of type $(\De,\Lm)$ is a Lie algebra $\cL$
over $\base$ that has two compatible gradings, one by the root lattice $Q$ of $\De$ and the
other by $\Lm$, such that a list of natural axioms hold (see Definition \ref{def:Lietorus}).
In that case the rank of $\Lm$  is called the \emph{nullity} of $\cL$.
Lie tori were introduced by Yoshii in  \cite{Y3,Y4} and, in an equivalent form that we use here, by Neher
in \cite{Neh1}.

Centreless (zero centre) Lie tori are of fundamental importance in the theory
of extended affine Lie algebras (EALAs), where they are used as the starting point
for the construction of all EALAs  \cite{Neh2}.
Perhaps the best known example occurs in nullity 1.  In that
case, any centreless Lie torus is isomorphic to the derived algebra modulo its centre
of an affine Kac-Moody Lie algebra $\fg$ \cite{ABGP},
and the full affine algebra $\fg$ is constructed
from this Lie torus by the familiar process of affinization.

In this article, we focus our attention on
centreless Lie tori that are \emph{fgc} (finitely generated as modules over their centroids).
We do this for two reasons. First, it is these Lie tori that play an important role in the
study of multiloop  Lie algebras; and vice versa (see more about this in Section  \myref{sec:Lietori}).
Second, it is known that the fgc assumption excludes only one family of centreless Lie tori
(see the discussion preceding Theorem~\myref{thm:structure}).

The structure of fgc centreless Lie tori is now quite well understood, using work
of a number of authors over a period of almost 15 years.    However, the isomorphism problem,
by which we mean the problem  of determining when two such Lie tori are isomorphic, is much less  understood.
Note that here and subsequently, the term isomorphic means
isomorphic as (ungraded) algebras, unless mentioned to the contrary.

The isomorphism problem for fgc centreless Lie tori has been solved
in nullities 0, 1 and 2.  Indeed, in nullities 0 and 1, a solution follows from classical
conjugacy theorems for maximal split toral $\base$-subalgebras of finite dimensional simple Lie
algebras and affine Kac-Moody Lie algebras respectively.  (See Sections 5.4 and 6.3 in
\cite{ABP3}.) In nullity 2,
the problem was solved  in \cite{ABP3} as part of the classification of nullity 2
multiloop Lie algebras.  (See \cite[Cor. 10.1.3 and Thm. 13.3.1]{ABP3}.)  In this paper, we consider
the problem for arbitrary nullity.   As one might expect, our
approach is to look for isomorphism invariants.

In order to describe some of our results, we briefly outline the structure of this paper, which
begins in Sections \ref{sec:prelim}--\ref{sec:cssLT} with some basic definitions
and properties of Lie tori.

In Section \ref{sec:closureLT}, we investigate the central closure
$\tL$ of an fgc centreless Lie torus $\cL$ of type $(\De,\Lm)$, which is obtained
from $\cL$ by extending the base ring from the centroid $C$ of $\cL$ to its
quotient field $\tC$.
It is known that $\tL$ is a finite dimensional isotropic central
simple Lie algebra over $\tC$, and hence the
theory of such Lie algebras can be brought to bear on our problem.
The main result in this section, Theorem \ref{thm:max}, describes
an explicit maximal split toral $\tC$-subalgebra $\tfh$ of $\tC$.
From this we  deduce Corollary \ref{cor:relativetype}, which
asserts that the relative type of $\tL$ is the type of the given root system
$\De$.  We note that Corollary \ref{cor:relativetype} was a
basic tool in the article   \cite{ABP3} mentioned above, but its proof
was left to be presented in this article.

In Section \ref{sec:inv},
we show that an fgc centreless Lie torus $\cL$ of type $(\De,\Lm)$
has four isomorphism invariants:
(i) the   type of the root system $\De$, which is called the \emph{root-grading type} of $\cL$; (ii)  the nullity
of $\cL$;  (iii) the rank of $\cL$ as a module over its centroid $C$, which
is called the \emph{centroid rank}  of $\cL$; and
(iv) a vector of positive integers, called the
\emph{root-space rank vector}  of $\cL$, that lists the ranks over $C$
of the root spaces of $\cL$ in the $Q$-grading.
Indeed,   the invariance of the centroid rank is clear.  However,
the other three quantities are defined using the graded structure of $\cL$
and hence their invariance requires more argument.
We establish the
invariance of the root-grading type and the root-space rank vector
using the results of Section~\ref{sec:closureLT}.  We also see that
invariance of the nullity follows easily from known facts about Lie tori.

We note that the four  invariants just discussed are rational, by which we mean, as in
\cite{Se}, that they are
defined without using base ring extension.  We also note that,
up to this point in the paper, our methods are
elementary, using for the most part
linear algebra, $\spl_2$-theory and
facts from \cite[Chap.~I]{Se} about finite dimensional central simple Lie
algebras.  For another approach, see \cite{P2}, \cite{GP1} and \cite{GP2},
where tools from Galois cohomology are used to study the isomorphism problem for forms
of algebras over Laurent polynomial rings and in particular for multiloop  Lie algebras.

In Section \ref{sec:isotopy},  we recall an equivalence relation for Lie tori, called \emph{isotopy}, that is finer than isomorphism as it takes into account the grading
\cite{ABFP2,AF}.
We observe that the group $\Lm/\Gm(\cL)$ is an isotopy invariant
(but not yet an isomorphism invariant) of  a centreless  Lie torus $\cL$, where $\Gm(\cL)$ denotes the $\Lm$-support of the centroid of $\cL$.  The main result of the section is a simple characterization of isotopy for centreless Lie tori.

In  the rest of the paper, we assume that  $\base$ is algebraically closed
and we apply the invariants from
Sections \ref{sec:inv} and \ref{sec:isotopy} to study classification and the isomorphism problem for fgc centreless Lie tori.
First in Section \ref{sec:structure} we summarize in one theorem the known structure theorems for fgc centreless Lie tori. It states that any such Lie torus is
either classical, which means roughly that it can be constructed as a special linear Lie algebra,
a special unitary Lie algebra, a special symplectic Lie algebra, or an orthogonal
Lie algebra over an associative torus; or it is
one of 27 Lie tori (defined for each sufficiently large nullity)  that we call exceptional.
Since the statements of the structure  theorems are spread over many papers,
we hope that our summary will be of independent interest to the reader.
Included in this section is a table, numbered as Table \ref{tab:exceptional}, of our invariants for exceptional Lie tori, with references to the literature.

In Section~\ref{sec:invclass}, we show how to calculate the invariants for classical Lie tori, and list the results in two tables, numbered as Tables \ref{tab:classical1} and \ref{tab:classical2}.  The three tables are then applied in
Section \ref{sec:isomorphism} to obtain results about the isomorphism problem for fgc  centreless Lie tori.
We show that the classes of exceptional and classical Lie tori have no overlap and that the four classes of classical Lie tori are similarly disjoint.  We then solve the isomorphism problem for special symplectic Lie tori and orthogonal Lie tori (the latter is easy), and we reduce the problem for exceptional Lie tori to consideration
of at most five particular algebras (in each nullity).  This reduces the classification of fgc centreless Lie tori to the separate isomorphism problems for (1) five particular exceptional Lie tori, (2) special linear Lie tori, and (3) special   unitary Lie tori.

In the final section, we discuss  these three problems under an additional conjugacy assumption
for certain (but not all) maximal split toral $\base$-subalgebras of an fgc centreless
Lie torus. The additional assumption is reasonable
since work in progress by Chernousov, Gille and Pianzola \cite{CGP}
will show that it always holds  (see
Remark \ref{rem:CGP}).   Under the conjugacy assumption,
we show that isotopy and isomorphism coincide for fgc centreless Lie tori
and use this to
complete the classification of exceptional Lie tori.  Also under the
conjugacy assumption, we complete the classification of special linear
Lie tori, leaving only the isomorphism problem for special unitary Lie tori
to be solved.

Finally, we note that the conjugacy assumption could have been used earlier
in the paper to demonstrate the invariance of the root-grading type and the root-space rank vector.
However, we did not do that since
we understand that \cite{CGP}  uses deep results from the  theory of group-schemes, whereas our goal
has been to deduce as much as possible about the isomorphism problem for
Lie tori using self-contained and elementary methods.

\emph{Acknowledgments.}  First, we thank Arturo Pianzola for carefully reading
an earlier version of this paper and making several suggestions that
substantially improved its presentation. We also thank him for keeping
us informed of the work in \cite{CGP}  on conjugacy.   Second, we thank the referee who noticed
and filled a small gap in the proof of Theorem \ref{thm:isotopychar}.  The referee also very helpfully suggested
the expansion, from the first version of the paper, of the material now included in Sections \ref{sec:invclass}
and~\ref{sec:isomorphism}.

\section{Preliminaries}
\label{sec:prelim}

Throughout the paper, \emph{we assume that $\base$ is a field of characteristic 0}.
Unless mentioned to the contrary, algebra will mean algebra over $\base$.

\myhead{The centroid} 
Suppose that $\cA$ is an algebra over $\base$.
The \emph{centroid }of $\cA$ is the  subalgebra of $\End_\base(\cA)$ defined by
\[\Cd_\base(\cA)  := \set{c\in \End_\base(\cA) \suchthat c(x\cdot y)
= c(x)\cdot y = x \cdot c(y) \text{ for } x,y\in \cA}.\]
Then $\base\, \id_\cA$ is a subalgebra of $\Cd_\base(\cA)$, which  we identify
with $\base$ in the evident fashion when $\cA \ne 0$.
The algebra $\cA$ is said to be \emph{central} if  $\Cd_\base(\cA) = \base\, \id_\cA$.

Note that $\cA$ is naturally a left $\Cd_\base(\cA)$-module; and we say that
$\cA$ is \emph{fgc} if this module is finitely generated.

The algebra $\cA$ is said to be \emph{perfect} if
$\cA \cdot \cA = \cA$, where $\cdot$ denotes the product in $\cA$.
If $\cA$ is perfect, then $\Cd_\base(\cA)$  is commutative.
If $\cA$ is simple (and hence perfect), then $\Cd_\base(\cA)$ is a field
and $\cA$ is a central simple algebra as an algebra over $\Cd_\base(\cA)$.

If $\cA$ is a  unital associative algebra, we denote the \emph{centre} of
$\cA$ by $Z(\cA)$.  Then the map
$z\mapsto \ell_z$ is an isomorphism of $Z(\cA)$ onto
$\Cd_\base(\cA)$, where $\ell_z \in \End_\base(\cA)$ is left multiplication
by $z$.

\begin{remark}
\label{rem:cinduced}

(i) If  $\cA$ is an algebra over an extension field $F$ of $\base$ and $\cA$ is perfect
(over $F$ or equivalently over $\base$), then $\Cd_\base(\cA) = \Cd_F(\cA)$.

(ii) Any isomorphism  $\ph: \cA \to \cA'$ of algebras induces a unique isomorphism
$\chi : \Cd_\base(\cA) \to \Cd_\base(\cA')$ such that
$\ph(cx) = \chi(c)\ph(x)$ for $c\in \Cd_\base(\cA)$, $x\in\cA$.
\end{remark}

\myhead{Involutions}
If $\cA$ is an algebra, an  \emph{involution} of
$\cA$ is an anti-automorphism ``$-$'' of $\cA$ (so $\overline{xy} = \bar y \bar x$ for $x,y\in\cA$)
of period 2.  In that case, we call $(\cA,-)$ an \emph{algebra with involution}.
If  the involution is fixed, we often use the notation
\begin{equation*}
\label{eq:notationA+-}
\cA_+ = \set{x\in \cA \suchthat \bar x = x} \andd
\cA_- = \set{x\in \cA \suchthat \bar x = - x},
\end{equation*}
in which case $\cA = \cA_+ \oplus \cA_-$.
If $\cA$ is unital and associative,
the   \emph{centre} of $(\cA,-)$ is defined as  $Z(\cA,-) := \set{x\in Z(\cA) \suchthat \bar x = x} = Z(\cA)\cap \cA_+$.

\myhead{Graded algebras}
If $\Lm$ be an abelian group and
$\cA = \bigoplus_{\lm\in\Lm} \cA^\lm$ is a $\Lm$-graded
algebra, we use the notation
$\supp_\Lm(\cA) := \set{\lm\in\Lm \suchthat \cA^\lm \ne \zero}$ for the
\emph{$\Lm$-support} of~$\cA$.

If $\cA$ is a $\Lm$-graded algebra and $\cA'$ is a $\Lm'$-graded algebra
we say that $\cA$ and $\cA'$ are \emph{isograded-isomorphic} if there
exists an algebra isomorphism $\ph : \cA \to \cA'$ and a group
isomorphism $\ph_\text{gr} : \Lm \to \Lm'$ such that
$\ph (\cA^\lm ) = {\cA'}^{\ph_\text{gr}(\lm)}$ for
$\lm\in \Lm$.

There   is an evident definition of a graded algebra with involution
(the involution is assumed to be graded) and of isograded-isomorphism for graded algebras with involution
(the map is assumed to preserve the involutions).

\myhead{Irreducible finite root systems}
As  in \cite{AABGP} and \cite{Neh1},
it will be convenient
for us to work with root systems that contain 0.
So, if $\cX$ is a finite dimensional vector space over $\base$,
by an \textit{irreducible finite root system} in $\cX$ we will mean
a finite subset $\De$ of $\cX$  such that $0\in \De$ and
$\Dec := \De \setminus\set{0}$ is
an irreducible finite root system in $\cX$ in the usual sense
(see \cite[chap.~VI, \S 1, D\'efinition~1]{Bo}).
We say that $\De$ is
\textit{reduced} if $2 \al \notin \Dec$ for $\al \in \Dec$.

An irreducible  finite root system $\De$ has one of the following types:
$\mathrm{A}_\ell\, (\ell \ge 1)$, $\mathrm{B}_\ell\,  (\ell \ge 2)$,
$\mathrm{C}_\ell\,  (\ell \ge 3)$, $\mathrm{D}_\ell\,  (\ell \ge 4)$,
$\mathrm{E}_6$, $\mathrm{E}_7$,
$\mathrm{E}_8$, $\mathrm{F}_4$ or $\mathrm{G}_2$  if $\De$
is reduced; or  $\mathrm{BC}_\ell\,  (\ell \ge 1)$ if $\De$
is not reduced.

We will use  the following notation  for an irreducible finite
root system $\De$ in $\cX$.  Let
\[ Q(\De) := \spann_\bbZ(\De)\]
be the \emph{root lattice} of $\De$.
Let $\cX^*$ denote the dual space of $\cX$,
let $\langle\ ,\ \rangle : \cX\times \cX^* \to k$
denote the natural pairing, and, if $\al\in \Dec$, let $\al^\vee$
denote the \textit{coroot} of $\al$ in $\cX^*$.  Finally, let
\[\Deic := \Dec\setminus 2\Dec\]
denote the set of \textit{indivisible}
nonzero roots in $\De$, and let $\Dei := \Deic\cup\set{0}$.
Then $\Dei$  is a  reduced irreducible finite root system
in $\cX$;  and, if $\De$ is reduced, we have~$\Dei~=~\De$.

\section{Split toral subalgebras and relative type} 
\label{sec:splittoral}

Suppose  that $\cL$ is a Lie algebra over $\base$.

A  \emph{split toral $\base$-subalgebra} of $\cL$ is an
abelian\footnote{It is not difficult to show that the abelian assumption
is superfluous (although we will not use this fact).}
$\base$-subalgebra $\fh$ of $\cL$ such
that there is a $\base$-basis for $\cL$ consisting
of simultaneous eigenvectors (with corresponding eigenvalues in $\base$)
for all of the operators $\ad(h)$,
$h\in \fh$.

If $\fh$ is a split toral $\base$-subalgebra of $\cL$, then  we have
the decomposition $\cL = \bigoplus_{\al\in \fh^*}\cL_\al$, called the \emph{root-space decomposition}
of $\cL$ relative to $\fh$, where
\[\cL_\al = \set{ x \in \cL \suchthat [h,x] = \al(h)x \text{ for } h\in \fh}\]
for $\al\in\fh^*$.
We set
\[\De_\base(\cL,\fh) := \set{\al\in \fh^* \suchthat \cL_\al \ne 0},\]
and we call $\De_\base(\cL,\fh)$ the \emph{root system} of $\cL$
relative to $\fh$.

The following formal result is well-known and easily checked
using Remark~\ref{rem:cinduced}.

\begin{lemma}
\label{lem:formal}
Suppose that $\cL$ (resp.~$\cL'$)
is a central perfect Lie algebra over a field $F$ (resp.~$F'$) that is an extension
field of $\base$.  Suppose that $\ph: \cL \to \cL'$ is a
$\base$-algebra isomorphism, $\fh$ is a split toral $F$-subalgebra
of $\cL$, and $\fh' = \ph(\fh)$. Then $\fh'$ is a split toral $F'$-subalgebra
of $\cL'$, which is maximal if and only if $\fh$ is maximal.
Moreover, setting $\De = \De_F(\cL,\fh)$, $Q = \spann_\bbZ(\De)$,
$\De' = \De_{F'}(\cL',\fh')$ and $Q' = \spann_\bbZ(\De')$,
there exists a unique group isomorphism $\rho : Q \to Q'$ such that
$\ph(\cL_\al) = \cL'_{\rho(\al)}$ for $\al\in Q$.  Furthermore,
$\rho(\De) = \De'$ and $\dim_F(\cL_\al) = \dim_{F'}(\cL'_{\rho(\al)})$
for $\al\in Q$.
\end{lemma}

A  finite dimensional central simple Lie algebra over $\base$ is said to be \emph{isotropic}
if it contains a nonzero split toral $\base$-subalgebra.

\begin{theorem}
\label{thm:Selig} \emph{\cite[\S I.2]{Se}} Suppose that $\cL$ is an isotropic finite dimensional central simple
Lie algebra over $\base$ and $\fh$ is a maximal split toral $\base$-subalgebra
of  $\cL$.  Then
\begin{itemize}
\item[(i)] $\De_\base(\cL,\fh)$ is an irreducible finite  root system
in $\fh^*$.
\item[(ii)] If $\fh'$ is another maximal split toral $\base$-subalgebra
of  $\cL$, there exists an automorphism $\ph$ of $\cL$ such that
$\ph(\fh) = \fh'$.
\end{itemize}
\end{theorem}

If $\cL$ is an isotropic finite dimensional central simple Lie algebra,
the \emph{relative type} of $\cL$ is defined to be the type of
the root system $\De_\base(\cL,\fh)$, where $\fh$ is a maximal split toral $\base$-subalgebra
of  $\cL$.  By  Theorem \ref{thm:Selig} and Lemma \ref{lem:formal} (with $F = F' = \base$)
this is independent of the choice of $\fh$.

\section{Lie tori}
\label{sec:Lietori}
For the rest of the paper \emph{we assume that $\De$ is an irreducible finite root system
with  $Q = Q(\De)$, and that $\Lm$ is a finitely generated free abelian group}.

This section contains the definition and some basic properties of Lie tori.
We restrict ourselves to the properties that we will need.  For the reader
wanting to learn more about this topic, two recent articles by Neher \cite{Neh3,Neh4}
are recommended.

In  order to recall the definition of a Lie torus,
we first introduce some  notation
for $Q\times \Lm$-gradings.
Let
\[\cL = \bigoplus_{(\al,\lm)\in Q\times\Lm}\dL{\al}{\lm}\]
be a $Q\times\Lm$-grading on a Lie algebra
$\cL$.\footnote{As is usual in the study of Lie tori,
it is convenient to use the notation $\dL{\al}{\lm}$ rather than $\cL^{(\al,\lm)}$ or $\cL_{(\al,\lm)}$
for the homogeneous component of degree $(\al,\lm)$.}
Then  $\cL = \bigoplus_{\al \in Q} \cL_\al$ is a $Q$-grading of $\cL$ with
\[\cL_\al := \bigoplus_{\lm\in \Lm} \dL{\al}{\lm} \qquad \text{ for }\al\in Q;\]
$\cL = \bigoplus_{\lm \in \Lm} \cL^\lm$ is a
$\Lm$-grading of $\cL$ with
\[\cL^\lm  := \bigoplus_{\al\in Q} \dL{\al}{\lm} \qquad \text{ for } \lm\in \Lm;\]
and we have $\dL{\al}{\lm} = \cL_\al \cap \cL^\lm$.
Conversely if $\cL$ has a $Q$-grading and a $\Lm$-grading
that are \emph{compatible} (which means that each $\cL_\al$ is a $\Lm$-graded
subspace of $\cL$ or equivalently that each $\cL^\lm$ is a $Q$-graded subspace of $\cL$),
then $\cL$ is $Q\times \Lm$-graded with
$\dL{\al}{\lm} = \cL_\al \cap \cL^\lm$.
From either point of view, we can simultaneously
regard $\cL$ as a $Q\times\Lm$-graded algebra, a $Q$-graded algebra and
a $\Lm$-graded algebra; and we correspondingly have the support sets
$\supp_{Q\times \Lm}(\cL)$, $\supp_Q(\cL)$ and
$\supp_\Lm(\cL)$.  We refer to the $Q$-grading as the \emph{root grading} of $\cL$,
and we refer to the  $\Lm$-grading as the \emph{external grading} of $\cL$.

\begin{definition} \cite{Neh1}
\label{def:Lietorus}
A \textit{Lie torus of type $(\Delta,\Lm)$} is a Lie algebra $\cL$ which has the following
properties:
\begin{itemize}
\item[(LT1)]  $\cL$ has a $Q\times\Lm$-grading
$\cL = \bigoplus_{(\al,\lm)\in Q\times\Lm}\dL{\al}{\lm}$ such that
$\supp_Q(\cL) = \De$.
\item[(LT2)]
\begin{itemize}
\item[(i)]  $(\Deic,0) \subseteq \supp_{Q\times \Lm}(\cL)$.
\item[(ii)]
If $(\al,\lm)\in \supp_{Q\times \Lm}(\cL)$ with $\al\in\Dec$,
then there exist elements $e_\al^\lm\in \dL{\al}{\lm}$
and $f_\al^\lm\in \dL{-\al}{-\lm}$ such that $\dL{\al}{\lm} = \base e_\al^\lm$,
$\dL{-\al}{-\lm} = \base f_\al^\lm$ and
\begin{equation}
\label{eq:LT2}
[[e_\al^\lm,f_\al^\lm],x_\beta] = \langle\beta,\al^\vee\rangle x_\beta
\end{equation}
for $x_\beta\in \cL_\beta$,  $\beta\in Q$.
\end{itemize}
\item[(LT3)] $\cL$ is generated as an algebra by the spaces $\cL_\al$, $\al\in \Dec$.
\item[(LT4)] $\Lm$ is generated as a group by $\supp_\Lm(\cL)$.
\end{itemize}
\end{definition}

In the definition given in \cite{Neh1}, it is only assumed   that
$\supp_Q(\cL) \subseteq \De$ in (LT1).
However, our stronger assumption is more convenient here and
it results in no loss of generality (see \cite[Remark 1.1.11]{ABFP2}).

If $\cL$ is a Lie torus, we assume (unless mentioned to the contrary)
that we have made a fixed choice of
a grading $\cL = \bigoplus_{(\al,\lm)\in Q\times\Lm}\dL{\al}{\lm}$ as in (LT1)
and elements $e_\al^\lm$ and $f_\al^\lm$ as in
(LT2)(ii). Thus if $(\al,\lm)\in \supp_{Q\times \Lm}(\cL)$ with $\al\in \Dec$,
then
$(e_\al^\lm, h_\al^\lm, f_\al^\lm)$
is an $\spl_2$-triple in~$\cL$, where  $h_\al^\lm = [e_\al^\lm,f_\al^\lm]$.
Hence the space $\cS_\al^\lm$ spanned by this triple is a $3$-dimensional split
simple Lie subalgebra of $\cL$.

\begin{remark}\label{rem:locfin}
If $(\al,\lm)\in \supp_{Q\times \Lm}(\cL)$ with $\al\in\Dec$,
then $\cL$ is a locally finite dimensional $\cS_\al^\lm$-module under the adjoint action.
Indeed, to see this it suffices to show that $U(\cS_\al^\lm) x_\beta$ is finite dimensional
for $x_\beta\in \cL_\beta$, $\beta\in \De$, where $U(\cS_\al^\lm)$ is the universal
enveloping algebra of $\cS_\al^\lm$.  This fact in turn follows from the
Poincar\'e-Birkhoff-Witt theorem for $\cS_\al^\lm$, \eqref{eq:LT2} and the assumption  that $\De$ is finite.
\end{remark}

\begin{definition}
\label{def:nulltype}
If $\cL$ is a Lie torus of type $(\De,\Lm)$, we define
the \emph{nullity} of $\cL$ to be $\rank_\bbZ(\Lm)$ and
the \emph{root-grading type} of $\cL$
to be the type of $\De$.
\end{definition}

We note that a Lie torus is perfect by \eqref{eq:LT2} and (LT3).

\begin{example}
\label{ex:untwisted}
Suppose that $\dot\fg$ is a finite dimensional split simple Lie algebra
with splitting Cartan subalgebra $\dot\fh$ over $\base$.  Let  $\De = \De_\base(\dot\fg,\dot\fh)$
and $Q = Q(\De)$; and let $\dot \fg = \bigoplus_{\al\in Q} \dot\fg_\al$ be the corresponding
root-space decomposition.
For $n\ge 0$, let
\[R_n  := \base[t_1^{\pm1},\dots,t_n^{\pm1}]\]
be the algebra of Laurent polynomials in $n$ variables over $\base$ with its natural $\bbZ^n$-grading
$R_n = \bigoplus_{\lm\in \bbZ^n} R_n^\lm$.
Then $\dot\fg \otimes R_n$ is an fgc centreless Lie torus of type $(\De,\bbZ^n)$ with
$(\dot\fg \otimes R_n)_\al^\lambda = \dot\fg_\al \otimes R_n^\lambda$ for $(\al,\lambda)\in Q\times \bbZ^n$.
We call $\dot\fg \otimes R_n$ the \emph{untwisted Lie torus} of type $(\De,\bbZ^n)$.
\end{example}

When $\base$ is algebraically closed, there is a twisted version of the above example
which constructs a subalgebra $\Lp(\dot\fg,\boldsg)$ of $\dot\fg \otimes R_n$
from a finite dimensional (split) simple Lie algebra $\dot\fg$
and an $n$-tuple $\boldsg$ of commuting finite order
automorphisms of
$\dot\fg$.\footnote{In \cite{ABFP2} and \cite{Na},
$\Lp(\dot\fg,\boldsg)$ is denoted by $M_{\boldm}(\dot\fg,\boldsg)$.}
The algebra $\Lp(\dot\fg,\boldsg)$ is   called
a \emph{nullity $n$ multiloop Lie algebra}.
If the common fixed
point algebra $\dot\fg^\boldsg$ is nonzero, then
$\Lp(\dot\fg,\boldsg)$ is an fgc
centreless Lie torus of nullity $n$ relative
to some  $Q\times\Lm$ grading on $\Lp(\dot\fg,\boldsg)$ \cite[Thm.~5.1.4]{Na}.
Conversely, any fgc centreless Lie torus of nullity
$n$ is isomorphic to $\Lp(\dot\fg,\boldsg)$ for some
$\dot\fg$ and $\sg$ as above with $\dot\fg^\sg \ne 0$ \cite[Thm.~3.3.1]{ABFP2}.

We will recall some
other constructions of Lie tori in Section \ref{sec:structure}.

We  now prove three lemmas about Lie tori using $\spl_2$-theory.
In each lemma \emph{we assume that $\cL$ is a Lie torus of type $(\De,\Lm)$}, where  we
recall that we are assuming that $\Lm$ is a finitely generated free abelian group.

The first lemma is an analogue for Lie tori of the well-known fact that
any associative $\Lm$-torus  is a domain.  (See Section
\myref{sec:structure} to recall the definition of an associative torus.)

\begin{lemma} \label{lem:domain}
If  $\al,\beta\in \Dec$ with $\pab < 0$,
$0\ne x_\al\in \cL_\al$ and $0\ne y_\beta\in\cL_\beta$,
then $\ad(x_\al)^{-\pab}y_\beta \ne 0$.
\end{lemma}

\begin{proof} Because of our assumptions on $\Lm$, we know
that we can give $\Lm$  a linear order (for example
the lexicographic order relative to some $\bbZ$-basis of $\Lm$).  Given nonzero $x\in \cL$,
this order on $\Lm$ allows us to speak of the nonzero component of \emph{highest degree} of $x$.

Suppose for contradiction that $\ad(x_\al)^{-\pab}y_\beta = 0$. Then
replacing $x_\al$ and $y_\beta$ by their nonzero components of highest
degree in the $\Lm$-grading,
we can assume that $x_\al \in \cL_\al^\lm$ and $y_\beta \in \cL_\beta^\mu$,
where $\lm, \mu\in \Lm$.
Thus, since the spaces $\cL_\al^\lm$ and $\cL_\beta^\mu$ are $1$-dimensional, we have
$\ad(e_\al^\lm)^{-\pab} e_\beta^\mu = 0$.
But,  by Remark \ref{rem:locfin},
$e_\beta^\mu$ lies in a finite dimensional $\cS_\al^\lm$-submodule of $\cL$. Further,
by \eqref{eq:LT2},
$e_\beta^\mu$ is an eigenvector for $\ad(h_\al^\lm)$ with eigenvalue $\pab<0$.
Therefore from the classification of finite dimensional
irreducible $\cS_\al^\lm$-modules, we have $\ad(e_\al^\lm)^{-\pab} e_\beta^\mu \ne 0$.
\end{proof}

The second lemma is an analogue for Lie tori of the well-known fact that
any invertible element in an associative $\Lm$-torus is homogeneous.

\begin{lemma} \label{lem:inverse}
Suppose $[x,y]\in \cL_0^0$, where
$0\ne x\in \cL_\al$, $0\ne y\in \cL_{-\al}$ and $\al\in \Dec$.
Then $x\in \cL_\al^\lm$ and $y\in \cL_{-\al}^{-\lm}$ for some $\lm\in\Lm$.
\end{lemma}

\begin{proof}  We order $\Lm$ as in the previous proof.
Let $x_\al^{\mu(x)}\in \cL_\al^{\mu(x)}$ be the nonzero $\Lm$-homogeneous
component of $x$ of highest degree $\mu(x)$, and
let $y_{-\al}^{\mu(y)} \in \cL_{-\al}^{\mu(y)}$
be the nonzero $\Lm$-homogeneous
component of $y$ of highest degree $\mu(y)$.
Then, $[x,y]-[x_\al^{\mu(x)}, y_{-\al}^{\mu(y)}]$ is
the sum of $\Lm$-homogeneous terms of degree less than
$\mu(x) + \mu(y)$.  But $[x_\al^{\mu(x)}, y_{-\al}^{\mu(y)}] \ne 0$
by Lemma \ref{lem:domain} with $\beta = -\alpha$.
So $\mu(x) = -\mu(y)$.  Similarly
if we use lowest degrees $\nu(x)$ and $\nu(y)$, we get
$\nu(x) = -\nu(y)$.  So
$\mu(x) = -\mu(y) \le -\nu(y) = \nu(x)$, which implies
that $x = x_\al^{\mu(x)}$.  Similarly,~$y = y_{-\al}^{\mu(y)}$.
\end{proof}

\begin{lemma} \label{lem:generate}  Suppose  $\cL$ is a Lie torus of type $(\De,\Lm)$.
If $\set{\al_1,\dots,\al_\ell}$
is a base for the root system $\De$, then the algebra $\cL$ is generated by
$\bigcup_{i=1}^\ell\left(\cL_{\al_i}\cup \cL_{-\al_i}\right)$.
\end{lemma}
\begin{proof} Let $\cM$ be the subalgebra of $\cL$ that is generated by the indicated set,
and let $E^\times  = \set{\al\in\Dec \suchthat \cL_\al \subseteq \cM}$.  In view of (LT3),
it suffices to show that $E^\times = \Dec$.  Now it follows from
\cite[(4)]{ABFP2} that $E^\times$ is stable under the action of the Weyl group
of $\De$.  Hence, $\Dei \subseteq E^\times$, and we are done if $\De$
is reduced.  Assume now that $\De$ is not reduced, and let $\al$ be a root of
smallest length in $\Dec$.  It remains to show that  $2\al \in E^\times$.
To verify this, it is enough
to show that $e_{2\al}^\sg \in \ad(e_\al^0) \cL_\al$
for all $\sg\in \Lm$.  This is an easy exercise using representations
of the algebra  $\cS_\al^0$.
We leave the details to the reader.
\end{proof}

\section{Centreless Lie tori}
\label{sec:cssLT}
In this section,   
\emph{we assume that $\cL$ is a centreless Lie torus of type
$(\De,\Lm)$} and 
we recall the basic facts  that we will need about $\cL$.  All of these facts were
announced  by Neher in \cite{Neh1} or \cite[\S 5.8(c)]{Neh3}.  For the convenience of the reader,
we  provide a proof or a reference for a proof in each case.

Set  
\[\fg = \cL^0 \andd \fh = \cL^0_0.\]
Then, by \cite[Prop. 1.2.2]{ABFP2},
$\fg$ is a finite dimensional split simple Lie algebra
with splitting Cartan subalgebra $\fh$.
Moreover [ibid],
$\De$ can be  uniquely identified
(by means of a linear  isomorphism of $\spann_\base(\De)$ onto $\fh^*$)
as a root system in $\fh^*$ in such a way that
\begin{equation*}
\Dei = \De_\base(\fg,\fh)
\end{equation*}
and $[e_\al^0,f_\al^0] =  \al^\vee$ for $\al\in \Deic$.
\emph{We will subsequently always make this identification}. In that case we have [ibid]
\begin{equation*}
[e_\al^\lm,f_\al^\lm] =  \al^\vee \quad
\text{for $(\al,\lm)\in \supp_{Q\times \Lm}(\cL)$, $\al\in\Dec$}
\end{equation*}
and
\begin{equation}
\label{eq:Lroot}
\cL_\al = \set{x\in \cL
\suchthat [h,x] = \al(h) x \text{ for } h\in \fh}\quad \text{for $\al\in Q$.}
\end{equation}
(Here $\al^\vee\in (\fh^*)^* = \fh$.)

Note that \eqref{eq:Lroot} tells us that $\fh$ is a split toral
$\base$-subalgebra of $\cL$ and that the  root grading
of $\cL$ is the root-space decomposition of $\cL$ relative
to  $\fh$.

Recall  that an algebra $\cA$  is said to be \emph{prime} if the product
of any two nonzero ideals of $\cA$ is nonzero.

\begin{proposition}
\label{prop:prime}
 $\cL$ is prime.
\end{proposition}

\begin{proof}  The main tool in the argument is Lemma \ref{lem:domain}, which tells us that
if  $\al,\beta\in \Dec$ with $\pab < 0$,
$0\ne x_\al\in\cL_\al$ and $0\ne y_\beta \in \cL_\beta$,
then
\begin{equation}
\label{eq:nonzeroprod}
0\ne \ad(x_\al)^{-\pab}y_\beta \in \cL_{w_\al}(\beta),
\end{equation}
where $w_\al$ is the reflection along $\al$ in the Weyl group $W$ of $\De$.

Suppose now that $\cI$ is a nonzero ideal of $\cL$.
By \eqref{eq:Lroot},
$\cI$ is $Q$-graded; that is $\cI = \bigoplus_{\al\in\De} \cI_\al$, where $\cI_\al = \cI \cap \cL_\al$.
Let  $\Dec(\cI) = \set{\al\in \Dec \suchthat \cI_\al \ne 0}$. 
We will see that $\Dec(\cI) = \Dec$.

Note first that
$\Dec(\cI) \ne \emptyset$. Indeed otherwise we have $\cI\subseteq \cL_0$, which implies $[\cI,\cL_\al] = 0$
for $\al\in \Dec$ and hence $[\cI,\cL] = 0$ by (LT3), contradicting our assumption that
$\cL$ is centreless.

We now claim that $W \Dec(\cI) \subseteq \Dec(\cI)$.  To see this,
it is enough to show that $w_\al(\beta)\in \Dec(\cI)$  for $\al \in \Dec$ and $\beta\in \Dec(\cI)$.
For this we can assume that $\pab < 0$ in which case our claim follows taking
$y_\beta\in \cI_\beta$ in \eqref{eq:nonzeroprod}.  Note that in particular,
if $\beta\in \Dec(\cI)$, we have $-\beta = w_\beta(\beta) \in \Dec(\cI)$.

Next we claim that $\Dec(\cI)$ and $\Dec\setminus \Dec(\cI)$ are orthogonal.  Indeed, if not, we can
choose $\al\in\Dec(\cI)$ and $\beta\in\Dec\setminus \Dec(\cI)$ with $\pab \ne 0$. Replacing,
$\al$ by $-\al$ if necessary, we can assume that $\pab <  0$.  But then taking
$x_\al \in \cI_\al$ in \eqref{eq:nonzeroprod}, we see that $w_\al(\beta)\in \Dec(\cI)$
and hence (by the previous claim) $\beta\in \Dec(\cI)$.  This contradiction proves the claim.
It then follows from the irreducibility of $\De$ that $\Dec(\cI) = \Dec$.

To prove the proposition, suppose for contradiction that $\cI$ and $\cJ$
are nonzero ideals of $\cL$ with
$[\cI,\cJ] \ne 0$.  Then $\Dec(\cI) = \Dec$ and $\Dec(\cJ) = \Dec$.  Hence, for
any $\al\in \Dec$, we have $\al\in \Dec(\cI)$ and
$-\al\in \Dec(\cJ)$.  So $\cI_\al \ne \zero$ and $\cJ_{-\al} \ne \zero$.
Since $[\cI_\al,\cJ_{-\al}]= 0$, this contradicts \eqref{eq:nonzeroprod} (with $\beta = -\al$).
\end{proof}

Let  $C = \Cd_\base(\cL)$. Then
$C = \bigoplus_{\lm\in\Lm} C^\lm$
is a $\Lm$-graded commutative associative algebra,  where
$C^\lm :=
\set{c\in C \suchthat c(\cL^\mu)\subseteq \cL^{\mu+\lm} \text{ for } \mu\in\Lm}$
\cite[Lemma 3.11(1)]{BN}.

Set
\[\Gm = \Gm(\cL) := \supp_\Lm(C).\]
Then  $\Gm$ is a subgroup of $\Lm$ [ibid], and
\begin{equation}
\label{eq:centstructure}
C \simeq \base[\Gm],
\end{equation}
as graded algebras, where $\base[\Gm]$ is the group
algebra of $\Gm$ with its natural $\Lm$-grading  \cite[Prop.~3.13(ii)]{BN}.

Recall (see Section \myref{sec:Lietori}) that $\Lm$ is called the external-grading
group of $\cL$.  Note also that $\cL$ it is  naturally graded by the quotient group
$\Lm/\Gm$, and we call the group $\Lm/\Gm$ the \emph{quotient external-grading group} of
$\cL$.

The following proposition follows from  \cite[Lemma 1.3.7 and Prop. 1.4.1]{ABFP2}:

\begin{proposition}  Suppose that $\cL$ is fgc.  Then
\label{prop:fincondLT}
\begin{itemize}
\item[(i)]   $\cL^\lm$ is finite dimensional for
$\lm\in \Lm$.\footnote{Part (i) is true without the assumptions that $\cL$ is fgc and centreless \cite[Thm. 5]{Neh2},
but the proposition as stated is all that we need.}
\item[(ii)]   $\Lm/\Gm$ is finite.
\end{itemize}
\end{proposition}

\goodbreak

\section{The central closure of an fgc centreless Lie torus}
\label{sec:closureLT}
In this section \emph{we assume that $\cL$ is an fgc centreless Lie torus of type
$(\De,\Lm)$} and we discuss  the central closure of $\cL$.  We continue using the notation
$\fh = \cL^0_0$,
$C = \Cd_\base(\cL)$ and $\Gm = \Gm(\cL)$ introduced in Section \ref{sec:cssLT}.

Taking  into account Proposition \ref{prop:fincondLT}, we now
fix a list $\lm_1,\dots,\lm_m$ of representatives
of the cosets of $\Gm$ in $\Lm$, with $\lm_1 = 0$.
For $\al\in \De$ and $1\le i \le m$, we choose a (finite) $\base$-basis $B_\al^i$
for $\cL_\al^{\lm_i}$.  For $\al\in \De$ we let $B_\al = \cup_{i=1}^m B_\al^i$;
and  we let $B = \cup_{\al\in\De} B_\al$.  Note that
$B$ is  finite since $\De = \supp_Q(\cL)$ is finite.

\begin{proposition}\
\label{prop:CmodLT}
\begin{itemize}
\item[(i)] If $\al\in Q$,
$\cL_\al$ is a $C$-submodule of $\cL$ and
$B_\al$ is a  $\Lm$-homogeneous $C$-basis for
$\cL_\al$.  Hence $\cL_\al$  is a free $C$-module of finite rank.

\item[(ii)] $B$ is a  $Q\times \Lm$-homogeneous $C$-basis for $\cL$.
Hence $\cL$ is a free $C$-module of finite rank.
\end{itemize}
\end{proposition}

\begin{proof} Since (ii) follows from (i), so we only need to check (i).
First, the fact that  $\cL_\al$ is a $C$-submodule of $\cL$
follows from \eqref{eq:Lroot}.   Also $B_\al$ is
$\Lm$-homogeneous by definition.   Finally,  the fact that $B_\al$
is a $C$-basis for $\cL_\al$ is easily checked directly using
\eqref{eq:centstructure}.
\end{proof}

The centroid $C$ of $\cL$ is an integral domain (for example
by \eqref{eq:centstructure}).  Let $\tC$ be the quotient
field of $C$, in which case $\tC$ is an extension field of $\base$.
Let
\[\tL := \tC \ot_C \cL.\]
Then $\tL$ is a algebra over $\tC$ which we call the \emph{central closure}
of $\cL$.

Now  $\cL$ is prime (by Proposition  \ref{prop:prime}), perfect and
fgc.  So
$\tL$ is a finite dimensional
central simple algebra over  $\tC$,
and the map $x \mapsto x\ot 1$ identifies
$\cL$ as a    $C$-subalgebra of  $\tL$.
(See for example \cite[\S 3]{ABP2.5}, which uses results from \cite[\S1]{EMO}.)

It follows from Proposition \ref{prop:CmodLT}(ii) that
$B$ is a $\tC$-basis of $\tL$ and hence
\begin{equation}
\label{eq:rankL}
\dim_\tC(\tL) = \rank_C(\cL).
\end{equation}

\begin{remark}
\label{rem:closureisom} If $\cL$ and $\cL'$ are fgc centreless Lie torus that
are isomorphic (as $\base$-algebras), it follows easily using
Remark \ref{rem:cinduced}(ii) that $\tL$ and $\tLp$
are isomorphic (as $\base$-algebras).
\end{remark}

Next  let 
\[\tfh = \tC\fh\]
in $\tC$.  It is clear that
$\tfh$ is a nonzero split toral $\tC$-subalgebra of
$\tL$, and hence $\tL$ is isotropic (see  Section \myref{sec:splittoral}).
We will show in Theorem \ref{thm:max} that $\tfh$ is a maximal split
toral $\tC$-subalgebra of~$\tL$.  

We first look at the root space  
decomposition of $\tC$ relative to $\tfh$.  For this,
let $\fh^* = \Hom_\base(\fh,\base)$ be the dual space of $\fh$ over $\base$ (as before),
and ${\tfh}^* = \Hom_\tC(\tfh,\tC)$ be the dual space of $\tfh$ over $\tC$.

\begin{proposition}\
\label{prop:rootspacec}
\begin{itemize}
\item[(i)]
$B_0^0$ is a $\base$-basis for $\fh = \cL_0^0$ and $B_0^0$ is a $\tC$-basis for
$\tfh$.  Hence $\dim_{\tC}(\tfh) = \dim_\base(\fh)$, and any
$\base$-basis for $\fh$ is a $\tC$-basis for $\tfh$.
\item[(ii)] There exists a unique
$\base$-linear map $\al\mapsto\tal$ of $\fh^*$ into $\tfh^*$ with
$\tal|_\fh = \al$ for $\al\in \fh^*$. Under this map, any
$\base$-basis for $\fh^*$ is sent to a $\tC$-basis for $\tfh^*$;
and we have
\begin{equation}
\label{eq:fhchar}
\fh = \set{\tildeh \in \tfh \suchthat \tal(\tildeh)\in \base \text{ for } \al\in \De}.
\end{equation}
\item[(iii)] Let $\tDe = \set{\tal \suchthat \al\in \De}$
and $\tQ =  \set{\tal \suchthat \al\in Q}$.
Then $\tDe$ is an irreducible finite root system over $\tC$ in $\tfh^*$ of the same type as
$\De$,\footnote{
In fact, one can check that $\tDe$ is isomorphic to the root system
obtained from $\De$ by base field extension from $\base$ to $\tC$ (as described
in \cite[Chap. VI, \S 1, Remark 1]{Bo}).}  and we have $\tQ = Q(\tDe)$.
\item[(iv)]  Let
$\tL_\tal := \set{\tx\in \tL \suchthat [\tildeh,\tx] = \tal(\tildeh)
\tx \text{ for } \tildeh\in \tfh}$ for $\al\in Q$.
Then $\tL_\tal = \tC \cL_\al$ for $\al\in Q$ and
$\tL =  \bigoplus_{\tal\in\tDe} \tL_\tal$.
\item[(v)] $\De_\tC(\tL,\tfh) = \tDe$.
\item[(vi)] If $\al\in \De$, then $B_\al$ is a $\tC$-basis for $\tL_\tal$ and hence
$\rank_C (\cL_\al) = \dim_\tC(\tL_\tal)$.
\end{itemize}
\end{proposition}

\begin{proof}  $B_0^0$ was chosen as a $\base$-basis for $\fh = \cL_0^0$, and
$B_0^0$ is part of the $\tC$-basis $B$ for $\tL$.  This implies (i); (ii) follows from (i)
and the fact that $\De$ contains a $\base$-basis of $\fh^*$;
and (iii) follows from (ii).

Next $\tL = \sum_{\al\in Q} \tC\cL_\al$ and $\tC \cL_\al \subseteq \tL_\tal$
for $\al\in Q$.  Since the sum
$\sum_{\tal\in\tQ}\tL_\tal$ is direct, this implies (iv).  Also,
if $\al\in Q$, we have
$\tL_\tal \ne \zero \iff \tC\cL_\al \ne 0 \iff \cL_\al \ne 0  \iff \al\in\De$.
(Here we have used the equality $\De = \supp_Q(\cL)$ from (LT1).)   So we have
(v). Finally, if $\al\in Q$, then
$B_\al$ is part of a $C$-basis for $\cL$ by Proposition \ref{prop:CmodLT},  so
(vi) follows from (iv).
\end{proof}

\begin{theorem}
\label{thm:max}
Suppose  that $\cL$
is an fgc centreless Lie torus of type $(\De,\Lm)$ with central closure
$\tL = \tC \cL$.  Let $\fh = \cL^0_0$ and $\tfh = \tC \fh$.  Then,
$\tfh$ is a maximal split toral $\tC$-subalgebra of $\tL$.
\end{theorem}

\begin{proof} We first claim that if $\al\in\Dec$ and $\tx$ is a nonzero element of $\tL_\tal$, then
$\ad(\tx)^2$ maps $\tL_{-\tal}$ bijectively onto $\tL_\tal$. Now
$\tL_{-\tal}$ and $\tL_{\tal}$ have the same dimension over $\tC$, since they
are paired by the Killing form of $\tL$ over $\tC$.
Hence to prove the claim
it is enough to show that $\ad(\tx)^2|_{\tL_{-\tal}}$ is injective.  For this, we argue by
contradiction. Suppose that $\ad(\tx)^2 \ty=0$ for some nonzero element $\ty$ of $\tL_{-\tal}$.
Now, by Proposition \ref{prop:rootspacec}(iv), $\tx = c^{-1} x$ and $\ty = d^{-1} y$,
where $c$ and $d$ are nonzero elements of $C$,
$0\ne x\in\cL_\al$ and $0\ne y\in \cL_{-\al}$.
Then $\ad(x)^2 y =0$.  But this contradicts Lemma \ref{lem:domain} (with $\beta = -\alpha$),
so we have the claim.

To prove the theorem, let $\ft$ be a maximal split toral $\tC$-subalgebra
of $\tL$ containing $\tfh$, and let $E = \De_\tC(\tL,\ft)$.  By Theorem \ref{thm:Selig}(i),
$E$ is an irreducible finite root system over $\tC$ in $\ft^*$.  We
choose a $\bbZ$-basis for the root lattice
$Q(E)$ of $E$ and order $Q(E)$ using the corresponding lexicographic order. For
$\al\in \De$ we let
\[E_\tal = \set{\ep \in E \suchthat \ep|_{\tfh} = \tal}.\]
Since $[\ft,\tfh] = 0$, we have $[\ft,\tL_\tal] \subseteq \tL_\tal$ for $\al\in \De$.
Then, since $\tfh \subseteq \ft$, it follows easily that
\begin{equation}
\label{eq:splitmore}
\tL_\tal = \bigoplus_{\ep\in E_\tal} \tL_\ep.
\end{equation}
for $\al\in \De$.  (Here $\tL_\tal$ denotes a root space relative to $\tfh$,
whereas $\tL_\ep$ denotes a root space relative to $\ft$.)

Now let $\al \in \Dec$.  Then, $E_\tal \ne \emptyset$ by
\eqref{eq:splitmore}.  Let $\ep$ be the
maximum root in $E_\tal$,
and fix nonzero $x\in \tL_\ep$.
Then, again by \eqref{eq:splitmore}, $x\in \tL_\tal$.  So, as we saw above,
$\ad(x)^2$ maps $\tL_{-\tal}$ bijectively onto $\tL_{\al}$.   It follows from this that
$E_\tal = E_{-\tal} + 2\ep$.  Since $E_{-\tal} = -E_\tal$, we have
$E_\tal = -E_{\tal} + 2\ep$.
Hence, if $\zeta \in E_\tal$, we have $\zeta = -\eta + 2\ep$ for
some $\eta \in E_\tal$, which gives $2\ep = \zeta + \eta$.
But if $\zeta < \ep$ this forces $2\ep < \ep + \eta \le 2\ep$, a contradiction.
Therefore $E_\tal = \set{\ep}$; that is $E_\tal$ is a singleton.

Finally, to show that $\ft \subseteq \tfh$, let $t\in \ft$.
Let $\set{\al_1,\dots,\al_\ell}$ be a base for the root system $\De$,
and choose $\ep_1,\dots,\ep_\ell$ in $E$ with $E_{\tal_i} = \set{\ep_i}$ for $1\le i \le \ell$.
But, by Proposition \ref{prop:rootspacec}(ii),
$\tal_1,\dots,\tal_\ell$ is a $\tC$-basis for $\tfh^*$, and so we can choose $h\in \tfh$
such that $\tal_i(h) = \ep_i(t)$ for $1\le i \le \ell$.
Then it follows from \eqref{eq:splitmore} (with $\al = \al_i$) that
$\ad(h) = \ad(t)$ on $\tL_{\tal_i}$ for each $i$.  Similarly, since
$E_{-\tal_i} = -E_{\tal_i}= \set{-\ep_i}$,
$\ad(h) = \ad(t)$ on $\tL_{-\tal_i}$ for each $i$.
So, by Lemma \ref{lem:generate}, $\ad(h-t) = 0$ on $\tL$.  Since $\tL$
is centreless, $t = h\in \tfh$.
\end{proof}

\begin{corollary}
\label{cor:splittoral} $\fh$ is a maximal split toral $\base$-subalgebra of $\cL$.
\end{corollary}

\begin{proof}  Suppose that $\ft$ is a split toral $\base$-subalgebra
of $\cL$ containing $\fh$.  Then $\tilde\ft := \tC\ft$ is a split toral
$\tC$-subalgebra of $\tL$ containing $\tfh$. Consequently, by Theorem
\ref{thm:max},~$\tilde \ft = \tfh$.

Now let $t\in \ft$.  So $t\in \tilde \ft = \tfh$.  But $\ad_\cL(t)$
is diagonalizable linear operator on $\cL$ over $\base$, and hence  $\ad_\tL(t)$ is
a diagonalizable linear operator on $\tL$ over
$\tC$ with eigenvalues lying in $\base$.  So
$\tal(t)\in \base$ for $\al\in \De$.  Thus, by \eqref{eq:fhchar},~$t\in \fh$.
\end{proof}

The next corollary was announced in \cite{ABP3} as Theorem 5.5.1
and used there as one of the main tools in the classification
of nullity 2 multiloop Lie
algebras.\footnote{In \cite{ABP3}, each result
in the sequence Theorem 5.5.1, Corollary 5.5.2, Theorem 9.2.1,
Theorem 12.2.1, Table 2, Theorem 13.2.1(b) and the
classification Theorem 13.3.1 uses its predecessor.}  
\footnote{The classification
of nullity 2 multiloop Lie algebras has subsequently also been obtained
by Gille and Pianzola in \cite{GP2} as a consequence of their
classification of $R_2$-loop simple adjoint groups and algebras using cohomological methods .}

\begin{corollary}
\label{cor:relativetype}
The relative type of $\tL$ is the root-grading type of $\cL$.
\end{corollary}

\begin{proof}  By Theorem \ref{thm:max}, the relative type of $\tL$ is
the type of the root system $\De_\tC(\tL,\tfh)$, which, by
Proposition \ref{prop:rootspacec}(iii) and (v), has the same type as $\De$.
\end{proof}

\section{Some isomorphism invariants}
\label{sec:inv}

Suppose that  
$\cL$ is an fgc centreless Lie torus $\cL$ of type
$(\De,\Lm)$ with centroid $C$.  We now describe four entities that we then show are
isomorphism invariants of $\cL$.

Recall first that
we defined the root-grading  type of $\cL$ and the nullity of $\cL$  in Definition
\ref{def:nulltype}.  Next, we define
the \emph{centroid rank} of $\cL$ to be
\[\crk(\cL) := \rank_C(\cL).\]

Finally,  it follows from \cite[(4)]{ABFP2} that if $\al,\beta\in\Dec$
are in the same orbit under the Weyl group of $\cL$,
then  $\rank_C(\cL_\al) = \rank_C(\cL_\beta)$.  Consequently,
this equality of rank holds whenever $\al,\beta$
have the same length.  So, we may define
$\rk_\short(\cL)$ to be $\rank_C(\cL_\al)$,
where $\al$ is a short
root\footnote{Our root length terminology follows \cite{AABGP}.
Roots of minimum length in $\Dec$ are called short,
roots in $\Dec\cap (2\Dec)$ are called extra-long, and all other roots in
$\Dec$ are called long.}
in $\Dec$.
If there exists a long root (resp.~an extra long root)
$\al$ in $\Dec$  we define $\rk_\lng(\cL)$
(resp.~$\rk_\extra(\cL)$) to be $\rank_C(\cL_\al)$.
Putting these quantities together, we define a vector of positive integers
\[\rkv(\cL) =
\left\{
  \begin{array}{ll}
    (\rk_\short(\cL))
            & \hbox{if $\De$ is reduced and simply laced,} \\
     (\rk_\short(\cL),\rk_\lng(\cL))
            & \hbox{if $\De$ is reduced and not simply laced,}\\
          (\rk_\short(\cL),\rk_\extra(\cL))
            & \hbox{if $\De$ is of type BC$_1$,} \\
     (\rk_\short(\cL),\rk_\lng(\cL),\rk_\extra(\cL))
            & \hbox{if $\De$ is of type BC$_\ell$, $\ell\ge 2$,}
  \end{array}
\right.
\]
which we call the \emph{root-space rank vector} of $\cL$.

\begin{proposition}
\label{prop:fourinv}
Suppose $\cL$  and $\cL'$ are fgc centreless Lie tori
with central closures $\tL$ and $\tLp$ respectively.
If
$\tL$ and $\tLp$ are isomorphic as Lie algebras over $\base$, then
\begin{itemize}
\item[(i)] The root-grading type of $\cL$ equals the root-grading type of $\cL'$.
\item[(ii)] The nullity of $\cL$ equals the nullity of $\cL'$.
\item[(iii)] $\crk(\cL) = \crk(\cL')$.
\item[(iv)] $\rkv(\cL) = \rkv(\cL')$.
\end{itemize}
\end{proposition}

\begin{proof}  We use the notation (for example $\fh = \cL_0^0$)
of Sections \ref{sec:cssLT}  and \ref{sec:closureLT}; and
we use corresponding primed notation (for example $\fh' = {\cL'}_0^0$)
for $\cL'$.  Let $\ph : \tL \to \tLp$ be a  $\base$-algebra
isomorphism.

(i): It follows from Lemma \ref{lem:formal} (with $F = C$ and $F' = C'$)
that $\tL$ and $\tLp$ have the same relative type.  Hence, by Corollary \ref{cor:relativetype},
we have (i).

(ii):  This is easy to see (and does not require
the results of Section \ref{sec:closureLT}).
Indeed, by Proposition \ref{prop:fincondLT}(ii), $\rank_\bbZ(\Lm) = \rank_\bbZ(\Gm)$
and similarly $\rank_\bbZ(\Lm') = \rank_\bbZ(\Gm')$.  So it suffices to show that
$\rank_\bbZ(\Gm) = \rank_\bbZ(\Gm')$.
Now $\tC = \Cd_\base(\tL)$ and $\tCp = \Cd_\base(\tLp)$, so $\tC \simeq \tCp$.
But, by \eqref{eq:centstructure},
$\tC$ (resp.~$\tCp$) is isomorphic to the field of rational functions
in $\rank_\bbZ(\Gm)$ (resp.~$\rank_\bbZ(\Gm')$) variables over $\base$, so
$\rank_\bbZ(\Gm) = \rank_\bbZ(\Gm')$  as desired.

(iii): This is clear (and does not use  Theorem \ref{thm:max}).  Indeed, it follows easily from
Remark \ref{rem:cinduced}(ii) (applied to $\tL$ and $\tLp$) that
$\dim_\tC(\tL) = \dim_\tCp(\tLp)$.  So,
by \eqref{eq:rankL},  $\rank_C(\cL) = \rank_{C'}(\cL')$.

(iv): By Theorem \ref{thm:max},
$\tfh$ is a maximal split toral $\tC$-subalgebra of $\tL$.
So, by Lemma \ref{lem:formal} applied to $\tL$ and $\tLp$,
$\ph(\tfh)$ is a maximal split toral $\tCp$-subalgebra
of $\tLp$.  Thus, by Theorem \ref{thm:Selig}(ii),
we can assume that $\ph(\tfh) = \tfhp$.  Now by Proposition
\ref{prop:rootspacec}(iii) and (v),
we have $\tQ = Q(\tDe)$ and $\De_\tC(\tL,\tfh) = \tDe$,
as well as corresponding equations for $\cL'$.
Thus, by Lemma \ref{lem:formal} applied to
$\tL$ and $\tLp$,
there exists a group isomorphism $\rho : \tQ \to \tQp$ such that
$\rho(\tDe) = \tDep$ and $\dim_\tC(\tL_\tal) = \dim_{\tCp}(\tLp_{\rho(\tal)})$
for $\tal\in \tQ$.
Finally, we let
$\tau : Q \to Q'$ be the group
isomorphism such that the following diagram commutes:
\[
\begin{CD}
Q       @>{\tau}>>    Q'\\
@VV {\widetilde{}} V            @VV {\widetilde{}} V\\
\tQ      @>{\rho}>>    \tQp
\end{CD}
\]
Then $\tau(\De) = \De'$; and we have
$\rank_C(\cL_\al) = \rank_{C'}(\cL'_{\tau(\al)})$ for $\al\in \De$
by  Proposition \ref{prop:rootspacec}(vi).
Finally, $\tau$ extends to a $\base$-linear isomorphism
$\fh^* \to {\fh'}^*$ which maps $\De$ onto $\De'$.  This extension is an isomorphism
of root systems, and so it maps  short roots, long roots
and extra long roots in $\Dec$ to roots of corresponding length in~${\De'}^\times$.
\end{proof}

By Remark \ref{rem:closureisom}, the following result follows immediately
from Proposition~\ref{prop:fourinv}.

\begin{theorem}
\label{thm:fourinv}
If $\cL$  and $\cL'$ are fgc centreless Lie tori that
are isomorphic as $\base$-algebras, then
(i), (ii), (iii) and (iv) in Proposition \ref{prop:fourinv} hold.  That is,
the root-grading type, the nullity,  the centroid rank,  and the root-space rank vector
are isomorphism invariants of an fgc centreless Lie torus.
\end{theorem}

The above proofs also show that the rank of $\cL_0$ over $C$
is an isomorphism invariant.  However,
this invariant is redundant, since it can be computed from the root-grading type,
the centroid rank and the root-space rank vector of $\cL$.

If $\cL$ is an  fgc centreless Lie algebra that possesses the graded structure
of a Lie torus, we can now unambiguously speak of the \emph{root-grading type},
the \emph{nullity}, the \emph{centroid rank} and the \emph{root-space rank vector}
of $\cL$, since these entities do not depend on the graded structure.

\begin{remark} \label{rem:Titsindex}
If  $\cL$ is an fgc centreless  Lie torus (or more generally any prime perfect fgc Lie algebra), the (Tits)
\emph{index} of $\cL$ is the index, as defined in \cite[\S 2.3]{T}, of the connected component of the automorphism
group of the finite dimensional central simple Lie algebra  $\tL$ over $\tC$. (See Section \myref{sec:closureLT} for the notation.)
The index of $\cL$ is a (non-rational) isomorphism invariant of $\cL$ \cite[Lemma 14.1.5]{ABP3}.
We won't use the index in this article. However, to provide a link to recent work on
multiloop algebras \cite{ABFP2, ABP3, GP1, GP2}, we will later display without proof  the index of
each fgc centerless Lie torus (see Table~\ref{tab:exceptional} and  Remark \ref{rem:Titsclassical}).
\end{remark}

\goodbreak
\section{Isotopy}
\label{sec:isotopy}

Suppose   that $\cL$ is a Lie torus of type $(\De,\Lm)$
and $\cL'$ is a Lie torus of type $(\De',\Lm')$.
An \emph{isotopy} of $\cL$ onto $\cL'$ is an algebra
isomorphism $\ph : \cL \to \cL'$ such that
\[\ph(\cL_\al^\lm) = {\cL'}_{\pr(\al)}^{\pe(\lm)+\ps(\al)},\]
for $\al\in Q$ and $\lm\in \Lm$,
where $\pr : Q \to Q'$ and $\pe : \Lm \to \Lm'$ are group isomorphisms
and $\ps : Q \to \Lm'$ is a group homomorphism.  In that case,
it is easy to check using (LT2)(i) and (LT4) that the maps $\pr$, $\pe$ and $\ps$ are uniquely determined.
It is also easy to check that the composite of two isotopies is an isotopy
and that the inverse of an isotopy is an isotopy.
We say that $\cL$ and  $\cL'$ are
\emph{isotopic}\footnote{The term isotopic was defined in a different way
in \cite[Def.~2.2.9]{ABFP2} and \cite[Def.~5.5]{AF}, but it is easy to check
that the definitions are equivalent.} if there
exists an isotopy from $\cL$ onto $\cL'$.

Finally, we define a \emph{bi-isomorphism}\footnote{Bi-isomorphism
is short for the more suggestive but cumbersome term bi-isograded-isomorphism.}
of $\cL$ onto $\cL'$ to
be an isotopy $\ph : \cL \to \cL'$ with $\ps = 0$.
If such a bi-isomorphism exists we say that $\cL$ and  $\cL'$ are
\emph{bi-isomorphic}.

If $\cL$ is bi-isomorphic to $\cL'$, then by definition $\cL$ is isotopic to $\cL'$;
however  the converse is not true \cite[Example 4.3.1]{ABFP2}.  Also, if
$\cL$ is isotopic to $\cL'$,  then by definition $\cL$ is isomorphic to $\cL'$.
We will consider the converse statement in
Section \myref{sec:conjimp}.

We next    show that
$\Lm/\Gm(\cL)$ is an isotopy invariant of a centreless Lie torus.

\begin{proposition}
\label{prop:isotopyinv}
Suppose that
$\cL$ and $\cL'$ are centreless Lie tori of type $(\De,\Lm)$ and
$(\De',\Lm')$ respectively.
If $\cL$ is isotopic to  $\cL'$, then $\Lm/\Gm(\cL) \simeq \Lm'/\Gm(\cL')$.
\end{proposition}

\begin{proof}  Let $\ph : \cL \to \cL'$ be an isotopy, $C = \Cd(\cL)$ and
$C' = \Cd(\cL')$.  Since $\ph$ is an isomorphism, we have
an induced isomorphism $\chi : C \to C'$
as in Remark \ref{rem:cinduced}(ii).  Then for $\lm, \mu\in \Lm$ and $\al\in Q$,
we have, setting $\lm' = \pe(\mu)+\ps(\al)$, that
\[\chi(C^\lm)({\cL'}_{\pr(\al)}^{\lm'})
= \chi(C^\lm) \ph(C_\al^\mu) = \ph(C^\lm \cL_\al^\mu) = \ph(\cL_\al^{\mu+\lm})
= {\cL'}_{\pr(\al)}^{\ph_e(\lm) + \lm'}.
\]
But for $\al\in Q$, $\pe(\Lm)+\ps(\al)  = \Lm'$.
Hence $\chi(C^\lm) \subseteq (C')^{\pe(\lm)}$ for $\lm\in \Lm$.  Thus,
since $\pe$ is invertible,
 $\chi(C^\lm) = (C')^{\pe(\lm)}$ for $\lm\in \Lm$.
Hence $\pe(\Gm(\cL)) = \Gm(\cL')$, and therefore $\pe$ induces the desired isomorphism.
\end{proof}

It does  not follow from Proposition \ref{prop:isotopyinv}  that
$\Lm/\Gm(\cL)$ is an isomorphism invariant.  We will consider this issue later in Section
\myref{sec:conjimp} for fgc centreless Lie tori.

We  have the following simple characterization
of isotopies of centreless  Lie tori.

\begin{theorem}
\label{thm:isotopychar}
Suppose that $\cL$ and $\cL'$ are centreless Lie tori of type $(\De,\Lm)$ and
$(\De',\Lm')$ respectively.  Let $\fh =  \cL_0^0$ and  $\fh' = {\cL'}_0^0$.
If $\ph : \cL \to \cL'$ is an algebra isomorphism, then
\[\ph \text{ is an isotopy} \iff \ph(\fh) = \fh'.\]
\end{theorem}

\begin{proof} The implication ``$\Rightarrow$'' is trivial.  To prove the reverse
implication, suppose that $\ph(\fh) = \fh'$.

We use the notation of
Section \ref{sec:cssLT} for $\cL$, and we
set
\[\Lm_\al = \set{\lm\in\Lm \suchthat \cL_\al^\lm \ne 0}\]
for $\al\in \Dec$.  We also use primed versions of
this notation for $\cL'$. Note  that if $\al\in \Dec$, then $\Lm_{-\al} = -\Lm_\al$
by LT2(ii).

Let  $\hat\ph : \fh^*\to (\fh')^*$ be the transpose of
$\ph^{-1}|_{\fh'} : \fh' \to \fh$.  Then, by \eqref{eq:Lroot},
$\ph(\cL_\al) = {\cL'}_{\hat \ph(\al)}$
for $\al\in \fh^*$.  So $\hat\ph(\De) = \De'$ and hence $\hat\ph(Q) = Q'$.
Let $\pr = \hat\ph|_Q : Q \to Q'$. Then
$\pr : Q \to Q'$ is a group isomorphism such that
$\pr(\De) = \De'$ (and hence also $\pr(\Dec) = {\De'}^\times$) and
\[\ph(\cL_\al) = {\cL'}_{\pr(\al)}\]
for $\al\in Q$.

Next let $\al\in\Dec$.  If $\lm\in \Lm_\al$, then
$0\ne e_\al^\lm\in \cL_\al$, $0\ne f_\al^\lm\in \cL_{-\al}$
and $[e_\al^\lm,f_\al^\lm]\in \fh$.
Thus, since $\ph(\fh) = \fh'$, we have
$0\ne \ph(e_\al^\lm)\in {\cL'}_{\pr(\al)}$, $0\ne \ph(f_\al^\lm)\in
{\cL'}_{-\pr(\al)}$
and $[\ph(e_\al^\lm),\ph(f_\al^\lm)]\in \fh'$.
So, by Lemma \ref{lem:inverse}, we have
$\ph(e_\al^\lm) \in \cLp_{\pr(\al)}^{\rho_\al(\lm)}$
and
$\ph(f_\al^\lm) \in \cLp_{-\pr(\al)}^{-\rho_\al(\lm)}$
for some
$\rho_\al(\lm) \in \Lmp_{\pr(\al)}$.  So counting dimensions, we have
$\ph(\cL_\al^\lm) = \cLp_{\pr(\al)}^{\rho_\al(\lm)}$ and
$\ph(\cL_{-\al}^{-\lm}) = \cLp_{-\pr(\al)}^{-\rho_\al(\lm)}$.  Since $\ph$
is an isomorphism, we have a bijection $\rho_\al: \Lm_\al \to \Lmp_{\pr(\al)}$
such that
\begin{equation}
\label{eq:char1}
\ph(\cL_\al^\lm) = \cLp_{\pr(\al)}^{\rho_\al(\lm)} \andd
\ph(\cL_{-\al}^{-\lm}) = \cLp_{-\pr(\al)}^{-\rho_\al(\lm)}
\end{equation}
for $\lm \in \Lm_\al$.

If $\al\in\Dec$ and $\lm\in \Lm_\al$, we have
$\ph(\cL_{-\al}^{-\lm}) = \cLp_{\pr(-\al)}^{\rho_{-\al}(-\lm)}$
since $-\lm \in -\Lm_\al = \Lm_{-\al}$.   Comparing this with the
second equation in \eqref{eq:char1}, we obtain
\begin{equation}
\label{eq:char1a}
\rho_{-\al}(-\lm) = -\rho_\al(\lm)
\end{equation}

We next claim that if $\al,\beta\in \Dec$, $\lm\in\Lm_\al$ and $\mu\in \Lm_\beta$,
we have\footnote{The equalities \eqref{eq:char1b} and \eqref{eq:char2b} are well-known (see for example \cite[\S 1.1]{ABFP2} and the earlier references there), but they arise naturally here so we give the arguments.}
\begin{equation}
\label{eq:char1b}\mu - \pab \lm \in \Lm_{w_\al(\beta)}
\end{equation}
and
\begin{equation}
\label{eq:char2}
\rho_{w_\al(\beta)}(\mu - \pab \lm) = \rho_\beta(\mu) - \pab \rho_\al(\lm).
\end{equation}
Indeed, this is clear if $\pab = 0$.  Next, suppose
$\pab < 0$. Then, by Lemma \ref{lem:domain}, we have
$0 \ne \ad(\cL_\al^\lm)^{-\pab} \cL_\beta^\mu \subseteq \cL_{w_\al(\beta)}^{\mu-\pab \lm}$,
which implies  \eqref{eq:char1b}.  Moreover, counting dimensions, we see that
$\cL_{w_\al(\beta)}^{\mu-\pab \lm} = \ad(\cL_\al^\lm)^{-\pab} \cL_\beta^\mu$.
Applying $\ph$ we get
\[\cL_{\ph_r(w_\al(\beta))}^{\rho_{w_\al(\beta)}(\mu-\pab \lm)}
=\ad(\cL_{\ph_r(\al)}^{\rho_\al(\lm)})^{-\pab} \cL_{\ph_r(\beta)}^{\rho_\beta(\mu)},\]
which implies  \eqref{eq:char2}.
Finally, if $\pab > 0$, then $\langle \beta,(-\al)^\vee\rangle < 0$
and $-\lm \in -\Lm_\al = \Lm_{-\al}$.  Hence, by our previous case, we have
\eqref{eq:char1b} and \eqref{eq:char2} with $\al$ replaced by $-\al$ and
$\lm$ replaced by $-\lm$, which gives \eqref{eq:char1b} and \eqref{eq:char2}
for $\al$ and $\lm$ using \eqref{eq:char1a}.  So we have the  claim.

To simplify notation, we now denote the reduced irreducible  finite  root system $\Dei$
by $E$. Let $W$ denote the Weyl group of $\De$ (= the Weyl group of $E$).
If $\al\in E^\times$, then $0\in \Lm_\al$ by LT(i).  So by \eqref{eq:char1b} (with $\lm = 0$),
we see that $\Lm_\beta \subseteq \Lm_{w_\al(\beta)}$ for $\al\in E^\times$, $\beta\in \Dec$.  Hence
$\Lm_\beta = \Lm_{w(\beta)}$
for $\beta\in \Dec$ and $w\in W$.
Thus
\begin{equation}
\label{eq:char2b} \Lm_\al = \Lm_\beta
\end{equation}
if $\al,\beta\in \Dec$ have the same  length.

Define $\sg: E^\times \to \Lm'$ by $\sg(\al) = \rho_\al(0)$.  Putting
$\lm= \mu = 0$ in \eqref{eq:char2}, we obtain
\begin{equation}
\label{eq:char3}
\sg(w_\al(\beta)) = \sg(\beta) - \pab \sg(\al)
\end{equation}
for  $\al,\beta\in E^\times$.
Let $\set{\al_1,\dots,\al_r}$ be a base for the root system $\De$, and choose
$\ps \in \Hom_\bbZ(Q,\Lm')$ such that $\ps(\al_i) = \sg(\al_i)$ for
$1\le i \le r$.  Define $\delta : E^\times \to \Lm$ by
$\delta(\al) = \sg(\al) - \ps(\al)$.  Then, since $\ps$ is $\bbZ$-linear, it follows
from \eqref{eq:char3} that
\begin{equation}
\label{eq:char4}
\delta(w_\al(\beta)) = \delta(\beta) - \pab \delta(\al)
\end{equation}
for   $\al,\beta\in E^\times$.
Now the set $X:=\set{\al\in E^\times \suchthat \delta(\al) = 0}$ contains
$\set{\al_1,\dots,\al_r}$; and so, by \eqref{eq:char4},
$X$ is stable under the action  of $W$.
Since $E$ is reduced, this implies that $X = E^\times$, so
$\sg(\al) = \ps(\al)$
for $\al\in E^\times$.  Hence
\begin{equation*}
\rho_{\al}(0) = \ps(\al)
\end{equation*}
for $\al\in E^\times$.

Next for $\al\in E^\times$, we define  $\tau_\al : \Lm_\al \to \Lm'$ by
\begin{equation}
\label{eq:char5}
\tau_\al(\lm) = \rho_\al(\lm) - \ps(\al).
\end{equation}
Observe that $\tau_\al(0) = 0$.

Suppose that  $\al,\beta\in E^\times$.
Then, since $\ps$ is $\bbZ$-linear,  we have $\ps(w_\al(\beta)) = \ps(\beta) - \pab \ps(\al)$.
Subtracting this from \eqref{eq:char2} we see that
\begin{equation}
\label{eq:char6}
\tau_{w_\al(\beta)}(\mu - \pab \lm) = \tau_\beta(\mu) - \pab \tau_\al(\lm)
\end{equation}
for $\lm\in \Lm_\al$, $\mu \in \Lm_\beta$.
Taking $\lm = 0$, we have $\tau_{w_\al(\beta)}(\mu) = \tau_\beta(\mu)$
for $\mu\in \Lm_\beta$.
Hence
\begin{equation}
\label{eq:char8}
\tau_{w(\beta)}= \tau_\beta
\end{equation}
for $\beta\in E^\times$ and $w\in W$.

Now fix a short root $\gamma$ in $E^\times$,
and let  $S = \Lm_\gamma$, which does not depend on the choice of $\gm$ by
\eqref{eq:char2b}.  It is known that
$0\in S$, $-S = S$, $S+2\Lm \subseteq S$,
$\Lm_\al \subseteq S$ for $\al \in E^\times$ and
$S$ generates the group $\Lm$
(see for example \cite[Lemma 1.1.12]{ABFP2}).
Hence $S$ contains a $\bbZ$-basis $\set{\nu_1,\dots,\nu_n}$ for $\Lm$
\cite[Prop.~II.1.11]{AABGP}.

We define $\tau : S \to \Lm'$ by $\tau = \tau_\gamma$,
which does not depend on the choice of $\gamma$
by \eqref{eq:char8}.
We claim next that
\begin{equation}
\label{eq:char9}
\tau_\al = \tau|_{\Lm_\al}
\end{equation}
for $\al$ in $E^\times$.
Indeed, if $\al$ has the same length as $\gamma$, we already know that \eqref{eq:char9} holds.
So we can assume that $\al$ is long and
$\langle \gamma, \al^\vee\rangle = -1$.  But then taking
$\beta= \gamma$ and $\mu = 0$ in \eqref{eq:char6},
we see that
$\tau_{w_\al(\gamma)}(\lm) = \tau_\al(\lm)$ for $\lm\in \Lm_\al$,
and so $\tau(\lm) = \tau_\al(\lm)$ for $\lm\in \Lm_\al$.

Next taking $\al = \gamma$ and $\beta = -\gamma$ in \eqref{eq:char6},
we see using \eqref{eq:char8} that
\begin{equation*}
\tau(\mu + 2\lm) = \tau(\mu) +2\tau(\lm)
\end{equation*}
for $\mu,\lm\in S$.

Define $\pe\in\Hom(\Lm,\Lm')$ by $\pe(\nu_i) = \tau(\nu_i)$
for $1\le i \le n$.  Further, define $\ep : S \to \Lm'$
by $\ep(\lm) = \tau(\lm)  - \pe(\lm)$.
Then $\ep(\nu_i)= 0$ for $1\le i \le n$ and
\begin{equation}
\label{eq:char10}
\ep(\mu + 2\lm) = \ep(\mu) +2\ep(\lm)
\end{equation}
for $\mu,\lm\in S$.  So, taking $\mu = -\lm$,
we have $\ep(-\lm) = -\ep(\lm)$ for $\lm\in S$.
Hence  $\ep(\pm \nu_i) = 0$ for $1\le i \le n$.

It follows by induction on $k$ using \eqref{eq:char10} that
$\ep(\mu + 2\sum_{i=1}^k\lm_i) = \ep(\mu) +2\sum_{i=1}^k\ep(\lm_i)$
for $\mu,\lm_1,\dots,\lm_k\in S$.  But each $\lm\in S$
is the sum of elements from $\set{\pm\nu_1,\dots,\pm\nu_n}$
and $\ep$ vanishes on the elements of this set. So we have
$\ep(\mu+2\lm) = \ep(\mu)$ for $\mu,\lm\in S$.
Therefore by \eqref{eq:char10}, $2\ep(\lm) = 0$ for $\lm\in S$,
and hence, since $\Lm$ has no 2-torsion, $\ep = 0$.
So $\tau(\lm) = \pe(\lm)$ for $\lm\in S$.
Thus, by \eqref{eq:char5} and \eqref{eq:char9}, we have
\begin{equation}
\label{eq:char11}
\rho_\al(\lm) = \pe(\lm) + \ps(\al).
\end{equation}
for $\al\in E^\times$, $\lm\in \Lm_\al$.
So by \eqref{eq:char1}, we have
\begin{equation}
\label{eq:char12}
\ph(\cL_\al^\lm) \subseteq \cLp_{\pr(\al)}^{\pe(\lm) + \ps(\al)}
\end{equation}
for  $\al\in E^\times$, $\lm\in \Lm_\al$.
But, by Lemma \ref{lem:generate},
every element of $\cL$ is the sum of products of elements
chosen from $\cL_\al^\lm$, $\al\in E^\times$, $\lm\in \Lm$.
So \eqref{eq:char12} holds for $\al\in Q$, $\lm\in\Lm$.

Finally, the isomorphism $\ph^{-1} : \cL' \to \cL$ satisfies an inclusion of exactly
the same form as \eqref{eq:char12}.  Using this it is easy to check that
$\pe : \Lm \to \Lm'$ is an isomorphism and hence
that equality holds in \eqref{eq:char11}
for $\al\in Q$, $\lm \in \Lm$.  We leave these arguments to the reader.
\end{proof}

\section{The structure of fgc centreless Lie tori}
\label{sec:structure}

For the rest of the article \emph{we assume that   $\base$ is algebraically closed}.

In this section,  we recall the structure theorems for fgc
centreless Lie tori. We combine these results into one theorem, which states that any fgc centreless Lie torus is either classical or exceptional.

Classical Lie tori and, in several cases, exceptional Lie tori are constructed from associative tori.  So we begin the section with a discussion of these graded algebras.

\myhead{Associative tori}
Recall \cite{Y1}  that an \emph{associative $\Lm$-torus}  (or simply
an associative torus) is a $\Lm$-graded unital associative algebra $\cA = \bigoplus_{\lm\in\Lm} \cA^\lm$
such that every $\cA^\lm$ is spanned by an invertible element for $\lm\in\Lm$.
(Equivalently, $\cA$ is a \emph{twisted group algebra} of $\Lm$ over $\base$.)
In that case, we call the rank of the group
$\Lm$  the \emph{nullity} of~$\cA$.

It is easy to check that if
$\cA$ is an associative $\Lm$-torus, $\cA'$ is an associative  $\Lm'$-torus,
and $\ph: \cA \to \cA'$ is an algebra isomorphism,
there exists a group isomorphism $\ph_\text{gr} : \Lm \to \Lm'$ such that
$\ph (\cA^\lm ) = {\cA'}^{\ph_\text{gr}(\lm)}$ for
$\lm\in \Lm$.  Thus it is not necessary to distinguish between isomorphism and
isograded-isomorphism for associative tori.

If $\cA$ is an associative $\Lm$-torus, we set $\Gm(\cA) := \supp_\Lm(Z(\cA))$.   Then
$\Gm(\cA)$ is a subgroup of $\Lm$ and $Z(\cA)$ is a commutative associative
$\Gm(\cA)$-torus.

It is  easily checked (and well-known) that any associative torus $\cA$
is a domain and hence prime (as a $\base$-algebra or equivalently as a ring).

The simplest example of an fgc associative torus
is the $\bbZ^n$-associative torus
$R_n = F[t_1^{\pm1},\dots,t_n^{\pm1}]$ with its natural $\bbZ^n$-grading.
(If $n=0$, $R_n = \base$ is graded by $\bbZ^0 = \set{0}$.)
Another important example is
obtained as follows.
Let $\zeta \in \base^\times$ and let
$\cQ(\zeta)$ be the algebra presented by the generators
$\qg_1^{\pm1},\qg_2^{\pm1}$ subject to the inverse relations
$\qg_i \qg_i^{-1} =  \qg_i^{-1} \qg_i = 1$, $i=1,2$, and the relation $\qg_1 \qg_2 = \zeta \qg_2 \qg_1$.
Then $\cQ(\zeta)$, with its natural $\bbZ^2$-grading,  is an associative $\bbZ^2$-torus which is fgc
if and only if $\zeta$ is a root of unity. We call $\cQ(\zeta)$ the \emph{quantum torus}
determined by~$\zeta$.

If $\cA_i$ is an associative $\Lm_i$-torus for $1\le i \le k$, then
$\cA_1\otimes \dots \otimes \cA_k$ is an associative $\Lm$-torus with
$\Lm =\Lm_1\oplus\dots\oplus\Lm_k$.  Moreover,
\[Z(\cA_1\otimes \dots \otimes \cA_k) = Z(\cA_1)\otimes \dots \otimes Z(\cA_k),\]
and $\cA_1\otimes \dots \otimes \cA_k$ is fgc if and only if each
$\cA_i$ is fgc.

Any fgc associative torus is isomorphic to a
tensor product
\begin{equation}
\label{eq:tensor}
\cA_1\otimes \dots \otimes \cA_k \otimes R_q,
\end{equation}
where  $k\ge 0$, $q\ge 0$ and $\cA_i \simeq \cQ(\zeta_i)$
with $\zeta_i$ a root of unity $\ne 1$ in $\base^\times$ for $i=1,\dots,k$.
Moreover, the $\zeta_i$'s can be  chosen satisfying further restrictions, and under
those restrictions Neeb has given
necessary and sufficient conditions for isomorphism  (or equivalently isograded-isomorphism) of two such
tensor products  \cite[Thm. 4.5]{Neeb} (although a subtle
point about determinants of certain integral matrices
is not resolved---see \cite[Conjecture 4.2]{Neeb}).

\myhead{Associative tori with involution}
An \emph{associative $\Lm$-torus with involution} is a $\Lm$-graded associative algebra with involution
$(\cA,-)$ such that $\cA$ is an associative  $\Lm$-torus.

If $(\cA,-)$ is an associative $\Lm$-torus with involution, we use the notation $\Gm(\cA,-) := \supp_\Lm(Z(\cA,-))$.  Then
$\Gm(\cA,-)$ is a subgroup of $\Lm$ and $Z(\cA,-)$ is a commutative associative
$\Gm(\cA,-)$-torus.  Also  we have
\begin{equation} \label{eq:Zi1}
Z(\cA) =  Z(\cA,-)\oplus (Z(\cA)\cap \cA_-),
\end{equation}
and we say that $(\cA,-)$ is of \emph{first kind} (resp.~\emph{second kind}) if
$Z(\cA) = Z(\cA,-)$ (resp. $Z(\cA) \ne Z(\cA,-)$).  If $(\cA,-)$ is of second kind,
then there exists a nonzero homogeneous element $s_0\in Z(\cA)\cap \cA_-$, and for any such $s_0$ we have
\begin{equation} \label{eq:Zi2} Z(\cA)\cap \cA_-  = s_0 Z(\cA,-) \andd \cA_- = s_0 \cA_+.
\end{equation}
Hence
\begin{equation} \label{eq:Zi3}
[\Gm(\cA) : \Gm(\cA,-)] = 1 \text{ or } 2
\end{equation}
according as
$(\cA,-)$ is of first or  second kind.

Four basic examples of associative tori with involution are
\begin{equation*}
(R_n,1),\quad (R_1,\natural),\quad (\cQ(-1),\natural) \andd (\cQ(-1),*),
\end{equation*}
graded by $\bbZ^n$, $\bbZ^1$, $\bbZ^2$ and $\bbZ^2$ respectively,
where the \emph{standard involution} $\natural$ of $R_1$ anti-fixes the generator $\qg_1$
($\qg_1^\natural = -\qg_1$); the \emph{standard involution} $\natural$ of $\cQ(-1)$
anti-fixes the generators $\qg_1$ and $\qg_2$;
and the \emph{reversal involution} $*$ of $\cQ(-1)$ fixes the
generators $\qg_1$ and $\qg_2$.\footnote{The term reversal involution
is used since $*$ reverses the order of products of the generators
$\qg_1^{\pm 1}$, $\qg_2^{\pm 1}$.  So $(\qg_1^{i_1}\qg_2^{i_2})^* = \qg_2^{i_2}\qg_1^{i_2} =
(-1)^{i_1i_2} \qg_1^{i_1}\qg_2^{i_2}$ for $i_1,i_2\in\bbZ$.}

If $(\cA_i,-)$ is an associative $\Lm_i$-torus with involution for $1\le i \le k$,  then
$(\cA_1,-)\otimes \dots \otimes (\cA_k,-)$ is an associative $\Lm$-torus with involution, where
$\Lm =\Lm_1\oplus\dots\oplus\Lm_k$; and we have
\[Z((\cA_1,-)\otimes \dots \otimes (\cA_k,-)) = Z(\cA_1,-)\otimes \dots \otimes Z(\cA_k,-).\]

Any associative torus with involution $(\cA,-)$
is  isomorphic (or equivalently iso\-graded-iso\-morphic) to a unique tensor product of the form
\begin{equation}
\label{eq:tensori}
(\cA_1,-) \otimes \dots \otimes (\cA_k,-) \otimes (\cA_{k+1},-) \otimes (R_q,1),
\end{equation}
where $k\ge 0$, $q\ge 0$, $(\cA_i,-) \simeq (\cQ(-1),\natural)$ for
$i=1,\dots,k$, and $(\cA_{k+1},-)$ is isomorphic to one of the associative tori
with involution
$(\base,1)$, $(R_1,\natural)$ or $(\cQ(-1),*)$
(see \cite[Thm.~2.7]{Y2} or \cite[Remark 5.20]{AFY}).
In that case
$(\cA,-)$ is of second kind if and only if
$(\cA_{k+1},-) \simeq (R_1,\natural)$.

We will use the following  lemmas about associative tori.

\begin{lemma} \label{lem:A+A+}
Suppose that $(\cA,-)$ is an associative torus with involution.
If $(\cA,-)$ is not isomorphic to
$(\cQ(-1),\natural)\otimes (R_q,1)$ for $q\ge 0$, then $[\cA_-,\cA_-] \subseteq \cA_+\cA_+$.
\end{lemma}

\begin{proof} Now  $(\cA,-)$ is isomorphic to an associative torus with involution of the form
\eqref{eq:tensori}.  If  $(\cA_{k+1},-) \simeq (R_1,\natural)$, then  $(\cA,-)$ is of second kind,
and choosing $s_0$ as in \eqref{eq:Zi2}, we have
$[\cA_-,\cA_-] = [s_0 \cA_-,s_0^{-1} \cA_-] \subseteq \cA_+ \cA_+$.  Also, if
$k=0$, then $[\cA_-,\cA_-]= 0$.

To complete the proof we assume that $k\ge 1$,
$(\cA_{k+1},-) \simeq (\base,1)$ or $(\cQ(-1),*)$,
and, if $k=1$,  $(\cA_{k+1},-) \simeq (\cQ(-1),*)$.  We show by induction that
\begin{equation}
\label{eq:induction}
\cA_- \cA_- = \cA \andd \cA_+ \cA_+ = \cA.
\end{equation}
First, if $k=1$,  then $(\cA,-) \simeq (\cQ(-1),\natural) \ot (\cQ(-1),*)\ot R_q$
and \eqref{eq:induction} is easily checked. Suppose next that $k\ge 2$.
When $(\cA,-) \simeq (\cQ(-1),\natural) \ot (\cQ(-1),\natural)\ot R_q$,
\eqref{eq:induction} is again easily checked.
Otherwise, we can identify $(\cA,-) = (\cB,-) \ot (\cC,-)$, where
$(\cB_,-) \simeq (\cQ(-1),\natural)$ and $(\cC,-)$ is of the form needed to apply our induction hypothesis.
Thus, $\cA_+ \cA_+ \supseteq (\cB_-\ot \cC_-) (\cB_-\ot \cC_-) = \cB_- \cB_- \ot \cC_-\cC_- = \cB\ot \cC = \cA$
and  $\cA_- \cA_- \supseteq (\cB_-\ot \cC_+) (\cB_-\ot \cC_+) = \cB_- \cB_- \ot \cC_+\cC_+ = \cB\ot \cC = \cA$.
\end{proof}

\begin{lemma} \
\label{lem:freemodule}
\begin{itemize}
\item[(i)] Suppose that  $\cA$ is an fgc associative $\Lm$-torus.
Then $[\Lm:\Gm(\cA)]$ is finite.  Further, if
$\cX$ is a graded $Z(\cA)$-submodule of $\cA$,
then $\cX$ is a free $Z(\cA)$-module of rank $\le [\Lm:\Gm(\cA)]$, with equality
holding if $\cX = \cA$.
\item[(ii)] Suppose that $(\cA,-)$ is an associative $\Lm$-torus with involution.
Then $\cA$ is fgc and $[\Lm:\Gm(\cA,-)]$ is finite. Further, if
$\cX$ is a graded $Z(\cA,-)$-submodule of $\cA$,
then $\cX$ is a free $Z(\cA,-)$-module of rank $\le [\Lm:\Gm(\cA,-)]$, with equality
holding if $\cX = \cA$.
\end{itemize}
\end{lemma}
\begin{proof}  i):  This is well-known (see \cite[Remark 4.4.2]{ABFP1} and the earlier references there), but we indicate a proof  for the convenience of the reader
and as a model for the proof of (ii).  Let $\cX$ be a graded $Z(\cA)$-submodule of $\cA$, and let
 $X =   \supp_\Lm(\cX)$.
Then  $\Gm(\cA) +X \subseteq X$.  Thus, $X$ is the union of cosets
of $\Gm(\cA)$ in $\Lm$, so we can choose  a set of representatives $\set{\mu_i}_{i\in I}$ of these cosets.
Further, choose $0\ne m_i\in \cA^{\mu_i}$ for $i\in I$.  Then $\set{m_i}_{i\in I}$ is
a $Z(\cA)$-basis for $\cX$,  so $\cX$ is a free $Z(\cA)$-module  of rank equal to the
cardinality of $I$.  In particular, $\cA$ is a free $Z(\cA)$-module  of rank
$[\Lm:\Gm(\cA)]$, which must therefore be finite since $\cA$ is fgc.

(ii):  The component associative tori in the tensor product decomposition  \eqref{eq:tensori} of $(\cA,-)$ are fgc,
and hence so is $\cA$.  So by (i), $[\Lm:\Gm(\cA)]$ is finite, and hence, by \eqref{eq:Zi3}
$[\Lm:\Gm(\cA,-)]$ is finite.  The rest of the proof of (ii) is similar to the proof   of~(i).
\end{proof}

\myhead{Classical Lie tori}
We next recall constructions of some fgc centreless Lie tori of root-grading type
$\type{A}{r},\ r\ge 1$;\  $\type{BC}{r}$ or $\type{B}{r},\  r \ge 1$;\
$\type{C}{r},\  r\ge 1$; and  $\type{D}{r},\ r \ge 4$ respectively.  Here types
$\type{B}{1}$ and $\type{C}{1}$ should be interpreted as $\type{A}{1}$, and
type $\type{C}{2}$  should be interpreted as $\type{B}{2}$.

In  each of these constructions, we use $\Mat_s(\cA)$ to denote the associative
algebra of $s\times s$ matrices over $\cA$ if $s\ge 1$ and $\cA$ an associative algebra.
Note that $\Mat_s(\cA)$ is therefore also a Lie algebra under the commutator product.
Furthermore $\Mat_s(\cA)$ is a free left $\cA$-module with basis
$\set{e_{ij}}_{1\le i,j\le s}$, where the action of $\cA$ on
$\Mat_s(\cA)$  is by left multiplication on entries and where $e_{ij}$  denotes the $(i,j)$-matrix unit.

In the last three constructions we will use the notation $J_p := (\delta_{i,p+1-j}) \in \Mat_{p}(\base)$,
for $p \ge 1$.  In other words,  $J_p$ is the $p\times p$ matrix with ones
on the anti-diagonal and zeroes elsewhere.

\begin{Constructions}\ 
\label{con:class}
\smallskip\par{\bf (A)}:
\cite[\S 2]{BGK}, \cite[\S 10]{AF}, \cite[\S 4.4]{Neh4}.\footnote{In some of the references
cited in this section,
additional assumptions (such as $\base = \bbC$ or $r\ge 2$)
are made that can  be checked to be unnecessary for our purposes.}
Suppose that $r\ge 1$ and
$\cA$ is an fgc associative $\Lm$-torus.  Let
$\cL = \spl_{r+1}(\cA)$ be the derived algebra of the Lie algebra
$\Mat_{r+1}(\cA)$ under the commutator product.
More explicitly, one easily checks that
\begin{equation}
\label{eq:spl}
\cL = \spl_{r+1}(\cA) = \set{X\in \Mat_{r+1}(\cA) \suchthat \trace(X) \in [\cA,\cA] },
\end{equation}
where $[\cA,\cA]$ is the space spanned by commutators in $\cA$.
Let $\fh = \sum_{i=1}^{r} \base (e_{ii} - e_{i+1,i+1})$.  Then
$\fh$ is a split toral $\base$-subalgebra of $\cL$ with irreducible finite root system
$\De = \Delta_\base(\cL,\fh)$ of type $\type{A}{r}$.
Moreover $\cL$ is an fgc centreless Lie torus of type $(\De,\Lm)$,
where the $Q$-grading of $\cL$ is the root-space decomposition relative to
$\fh$ and the $\Lm$-grading of $\cL$ is induced by the $\Lm$-grading of~$\cA$.
We call $\cL$ the
\emph{$(r+1)\times (r+1)$-special linear
Lie torus} over $\cA$.

\smallskip\par{\bf (BC--B)}: \cite[\S III.3]{AABGP}, \cite[\S 7.2]{AB}.
Suppose that $r\ge 1$,
$L$ is a finitely generated free abelian group (which we will embed in a
larger group $\Lm$ of the same rank below), and
$(\cA,-)$  is an associative $L$-torus with involution.  Suppose also that $m\ge 1$ and
$D = \diag(d_1,\dots,d_m) \in \Mat_{m}(\cA)$, where
$d_1,\dots,d_m$ are nonzero  homogeneous hermitian elements of $\cA$ whose
respective degrees $\delta_1,\dots, \delta_m$ in $L$ are distinct modulo $2L$ with
$d_1  = 1$ and $\delta_1 = 0$.
To eliminate overlap with the other constructions, we assume that if  $r=1$ and $- = 1$,
then $m\ge 5$.
Let $G = \diag(J_{2r},D)$ in block diagonal form, and let
$\cL = \spu_{2r+m}(\cA,-,D)$  be the  derived algebra of the Lie algebra
$\set{X\in \Mat_{2r+m}(\cA) \suchthat G^{-1}\bar X ^t G = -X}$ under the commutator product.
More explicitly we have  \cite[\S 7.2.3]{AB}
\[
\cL = \spu_{2r+m}(\cA,-,D) =
\set{X\in \Mat_{2r+m}(\cA) \suchthat G^{-1}\bar X ^t G = -X,\ \trace(X) \in [\cA,\cA] }.
\]
To describe the external grading on $\cL$, we first embed
$L$ in the rational vector space $\bbQ \otimes_\bbZ L$ and let
$\Lm$ be the subgroup of $\bbQ \otimes_\bbZ L$ generated by
$L$ and $\frac 12 \delta_1,\dots, \frac 12 \delta_m$.
Further we define $\tau_i\in L$ for $1\le i \le 2r+m$
by $\tau_i = 0$ for $1\le i \le 2r$ and $\tau_{2r+i}  = \delta_i$ for $1\le i \le m$.
Then the associative algebra  $\Mat_{2r+m}(\cA)$ is $\Lm$-graded by assigning the degree
$\lm + \frac 12 \tau_i - \frac 12 \tau_j$ to each element in $\cA^\lm e_{ij}$
for $\lm\in L$, $1 \le i,j\le 2r+m$;  and one checks directly that
the involution $X\mapsto G^{-1}\bar X ^t G$ of $\Mat_{2r+m}(\cA)$ is $\Lm$-graded.
Consequently, the Lie algebra
$\Mat_{2r+m}(\cA)$ under the commutator product is $\Lm$-graded, and
$\cL$ is a $\Lm$-graded subalgebra of this algebra.
To describe the root grading on $\cL$, let $\fh = \sum_{i=1}^{r} \base (e_{ii} - e_{2r+1-i,2r+1-i})$.  Then
$\fh$ is a split toral $\base$-subalgebra of $\cL$ with irreducible finite root system
$\De = \Delta_\base(\cL,\fh)$, and the type of $\De$ is $\type{BC}{r}$ if
$- \ne 1$ and $\type{B}{r}$ if $- = 1$.
Also, the root-space decomposition of $\cL$ relative to $\fh$
is a $Q$-grading of $\cL$ which is compatible with the $\Lm$-grading just described.
With the resulting $Q\times \Lm$-grading, $\cL$ is an fgc centreless Lie torus of type
$(\De,\Lm)$.\footnote{As an example, if we take $(\cA,-) = (R_2,1)$, $r= 1$
and $D = \diag(1,t_1,t_2)$, $\spu_{5}(\cA,-,D)$ is the centreless Lie torus
whose universal central extension is called the
\emph{baby TKK algebra} in \cite{Tan}.}
We call $\cL$ the
\emph{$(2r+m)\times (2r+m)$-special unitary Lie torus}
over $(\cA,-)$ determined by $D$.

\smallskip\par{\bf (C)}:  \cite[\S III.4]{AABGP}, \cite[\S 11]{AF}.
Suppose that $r\ge 1$ and
$(\cA,-)$  is an  associative $\Lm$-torus with involution.
To avoid degenerate cases and
eliminate overlap with the other constructions, we assume that if  $r=1$ or $2$, then
$(\cA,-)$ is not isomorphic
to $(R_q,1)$, $(R_1,\natural)\otimes (R_{q},1)$ or
$(\cQ(-1),\natural)\otimes (R_{q},1)$ for $q\ge 0$.
Let $G = \left[\begin{smallmatrix} 0&J_r  \\-J_r&0 \end{smallmatrix}\right]
\in\Mat_{2r}(\base)$ in block form,
and let $\cL = \ssp_{2r}(\cA,-)$ be the  derived algebra of the Lie algebra
$\set{X\in \Mat_{2r}(\cA) \suchthat G^{-1}\bar X ^t G = -X}$ under the commutator
product.  Once again, we have more explicitly that
\[
\cL = \ssp_{2r}(\cA,-) =
\set{X\in \Mat_{2r}(\cA) \suchthat G^{-1}\bar X ^t G = -X,\ \trace(X) \in [\cA,\cA] }.
\]
Indeed, if $r\ge 2$ this is easily checked directly, whereas if $r= 1$ it is easily  checked using
Lemma \ref{lem:A+A+}.
Let $\fh = \sum_{i=1}^{r} \base (e_{ii} - e_{2r+1-i,2r+1-i})$.  Then
$\fh$ is a split toral $\base$-subalgebra of $\cL$ with irreducible finite root system
$\De = \Delta_\base(\cL,\fh)$ of type $\type{C}{r}$ (see  the proof of
Proposition \ref{prop:classinv} below for this calculation), and
$\cL$ is an fgc centreless Lie torus of type $(\De,\Lm)$ with gradings
determined by $\fh$ and $\cA$ as in (A) above.
We call $\cL$ the
\emph{$(2r)\times (2r)$-special symplectic Lie torus}
over $(\cA,-)$.

\smallskip\par{\bf (D)}:
Suppose that $r\ge 4$ and
$\cA = R_n$ with its natural grading by $\Lm = \bbZ^n$.
Let
\[
\cL = \orth_{2r}(\cA) :=
\set{X\in \Mat_{2r}(\cA) \suchthat J_{2r}^{-1} X ^t J_{2r} = -X}.
\]
Then $\cL$ is a Lie algebra under the commutator product.
Let $\fh = \sum_{i=1}^{r} \base (e_{ii} - e_{2r+1-i,2r+1-i})$.  Then
$\fh$ is a split toral $\base$-subalgebra of $\cL$ with irreducible finite root system
$\De = \Delta_\base(\cL,\fh)$ of type $\type{D}{r}$, and
$\cL$ is an fgc centreless Lie torus of type $(\De,\Lm)$ with gradings
determined by $\fh$ and $\cA$ as in (A) above.  In fact,
$\cL \simeq \orth_{2r}(\base) \otimes R_n$ is just the untwisted   Lie torus
of type $(\De,\bbZ^n)$ (see Example \ref{ex:untwisted}), viewed as an algebra of matrices.
We call $\cL$ the
\emph{$(2r)\times (2r)$-orthogonal Lie torus}
over $\cA$.
\end{Constructions}

We note  that in each of the constructions, the indicated subalgebra $\fh$
is the  maximal split toral $\base$-subalgebra $\cL_0^0$ of $\cL$ that was denoted by $\fh$ in
Sections \ref{sec:cssLT} to~\ref{sec:isotopy}.

We call an fgc centreless Lie torus that arises from any one of the Constructions
(A), (BC--B), (C) or (D)  a \emph{classical Lie torus}.

\begin{remark} If we allow $r=0$
in Constructions (A) and (BC--B), we obtain multiloop Lie algebras
that are not Lie tori.\footnote{A few low rank cases must be excluded but these are easy to
identify.}  Indeed, one can show that they are multiloop Lie algebras (see the discussion following Example
\ref{ex:untwisted})
using the multiloop realization theorem
\cite[Cor.~8.3.5]{ABFP1}.  (The hypotheses of that theorem can be checked using a base ring extension argument as in Proposition \ref{prop:centroidLT} below.)
Also, one can show that they do not contain nonzero split toral $\base$-subalgebras
(using  \cite[\S 4.5.9]{ABP3} and \cite[Prop. 5.2.5]{AB}), which shows that they are not Lie tori.
Since our interest
in this article is in Lie tori, we omit the details in this remark and we do not consider the  $r=0$ case
further.
\end{remark}

\myhead{Exceptional Lie tori}  We next  
display in Table \ref{tab:exceptional} a list of
fgc centreless Lie tori that we call \emph{exceptional Lie tori}.
For convenience of reference we have labeled
these Lie tori  as \#1--27 in the column  labeled $\#$.
Each row of the table represents exactly one
Lie torus of nullity $n$ for each $n\ge n_0$, where
the minimum nullity $n_0$ is displayed in the second column (not counting the
$\#$ column) of the
table.\footnote{For the Lie tori numbered 25, 26 and 27,
there are parameters $\sg_0$ and $\mu$ in the description given
in \cite{F}.  However, one can argue as in \cite[\S 10]{AF}, that
the Lie torus does not depend on these choices up to bi-isomorphism.}


\begin{table}[ht]
\relscale{.87}
\renewcommand{\arraystretch}{1.2}
\setlength\doublerulesep{1pt}
\centering

\begin{tabular}
{p{.3cm} ||   p{1.2cm} | p{.4cm}| p{.9cm} | p{1.5cm}| p{1.2cm} | p{.8cm} | p{4cm} }


\drop\drop $\#$
&Root-grading type
& \drop\drop $n_0$
& \drop\drop  $\crk(\cL)$
& \drop\drop  $\rkv(\cL)$
& \drop\drop $\Lm/\Gamma(\cL)$
& \drop\drop Index
& \drop\drop Reference
\\
\hline\hline

$1$
&$\type{A}{1}$
& $3$
& $133$
& $(27)$
& $\bbZ_3^3$
& $\Tindex{}{E}{7,1}{78}$
&\cite[Example~6.8(3)]{Y1}
\\
\hline

$2$
&$\type{A}{2}$
& $3$
& $78$
& $(8)$
& $\bbZ_2^3$
& $\Tindex{1}{E}{6,2}{28}$
&\cite[Example~9.2]{AF}
\\
\hline

$3$
&$\type{C}{3}$
& $3$
& $133$
& $(8,1)$
& $\bbZ_2^3$
& $\Tindex{}{E}{7,3}{28}$
&\cite[Thm.~4.87(ii)]{AG}
\\
\hline

$4$
&$\type{E}{6}$
& $0$
& $78$
& $(1)$
& $\set{0}$
& $\Tindex{1}{E}{6,6}{0}$
&untwisted
\\
\hline

$5$
&$\type{E}{7}$
& $0$
& $133$
& $(1)$
& $\set{0}$
& $\Tindex{}{E}{7,7}{0}$
&untwisted
\\
\hline

$6$
&$\type{E}{8}$
& $0$
& $248$
& $(1)$
& $\set{0}$
& $ \Tindex{}{E}{8,8}{0}$
&untwisted
\\
\hline

$7$
&$\type{G}{2}$
& $0$
& $14$
& $(1,1)$
& $\set{0}$
& $\Tindex{}{G}{2,2}{0}$
&untwisted
\\
\hline

$8$
&\ditto
& $1$
& $28$
& $(3,1)$
& $\bbZ_3$
& $\Tindex{3}{D}{4,2}{2}$
&\cite[Thm.~5.63, p=1]{AG}
\\
\hline

$9$
&\ditto
& $2$
& $78$
& $(9,1)$
& $\bbZ_3^2$
& $\Tindex{1}{E}{6,2}{16}$
&\cite[Thm.~5.63, p=2]{AG}
\\
\hline

$10$
&\ditto
& $3$
& $248$
& $(27,1)$
& $\bbZ_3^3$
& $\Tindex{}{E}{8,2}{78}$
&\cite[Thm.~5.63, p=3]{AG}
\\
\hline

$11$
&$\type{F}{4}$
& $0$
& $52$
& $(1,1)$
& $\set{0}$
& $\Tindex{}{F}{4,4}{0}$
&untwisted
\\
\hline

$12$
&\ditto
& $1$
& $78$
& $(2,1)$
& $\bbZ_2$
& $\Tindex{2}{E}{6,4}{2}$
&\cite[Thm.~5.50, p=1]{AG}
\\
\hline

$13$
&\ditto
& $2$
& $133$
& $(4,1)$
& $\bbZ_2^2$
&  $\Tindex{}{E}{7,4}{9}$
&\cite[Thm.~5.50, p=2]{AG}
\\
\hline

$14$
&\ditto
& $3$
& $248$
& $(8,1)$
& $\bbZ_2^3$
& $\Tindex{}{E}{8,4}{28}$
&\cite[Thm.~5.50, p=3]{AG}
\\
\hline

$15$
&$\type{BC}{1}$
& $3$
& $52$
& $(8,1)$
& $\bbZ_2^3$
& $\Tindex{}{F}{4,1}{21}$
&\cite[Thm.~5.19(b), k=0]{AFY}
\\
\hline

$16$
&\ditto
& $4$
& $78$
& $(16,8)$
& $\bbZ_2^4$
& $\Tindex{2}{E}{6,1}{29}$
&\cite[Thm.~5.19(b), k=1]{AFY}
\\
\hline

$17$
&\ditto
& $5$
& $133$
& $(32,10)$
& $\bbZ_2^5$
& $\Tindex{}{E}{7,1}{48}$
&\cite[Thm.~5.19(b), k=2]{AFY}
\\
\hline

$18$
&\ditto
& $6$
& $248$
& $(64,14)$
& $\bbZ_2^6$
& $\Tindex{}{E}{8,1}{91}$
&\cite[Thm.~5.19(b), k=3]{AFY}
\\
\hline

$19$
&\ditto
& $5$
& $78$
& $(20,1)$
& $\bbZ_2^5$
& $\Tindex{2}{E}{6,1}{35}$
&\cite[Thm.~10.6(a), case 1]{AFY}
\\
\hline

$20$
&\ditto
& $6$
& $133$
& $(32,1)$
& $\bbZ_2^6$
& $\Tindex{}{E}{7,1}{66}$
&\cite[Thm.~10.6(a), case 2]{AFY}
\\
\hline

$21$
&\ditto
& $7$
& $248$
& $(56,1)$
& $\bbZ_2^7$
& $\Tindex{}{E}{8,1}{133}$
&\cite[Thm.~10.6(a), case 3]{AFY}
\\
\hline

$22$
&\ditto
& $5$
& $133$
& $(32,1)$
& $\bbZ_2^5$
& $\Tindex{}{E}{7,1}{66}$
&\cite[Remark 10.6(a)]{AFY}
\\
\hline

$23$
&\ditto
& $3$
& $133$
& $(32,1)$
& $\bbZ_2 \oplus \bbZ_4^2$
& $\Tindex{}{E}{7,1}{66}$
&\cite[Thm.~13.3, case 1]{AFY}
\\
\hline

$24$
&\ditto
& $3$
& $248$
& $(56,1)$
& $\bbZ_4^3$
& $\Tindex{}{E}{8,1}{133}$
&\cite[Thm.~13.3, case 2]{AFY}
\\
\hline

$25$
&$\type{BC}{2}$
& $3$
& $78$
& $(8,12,1)$
& $\bbZ_2^3$
& $\Tindex{2}{E}{6,2}{16'}$
&\cite[Lem.~7, $\tilde n=0$]{F}
\\
\hline

$26$
&\ditto
& $4$
& $133$
& $(16,16,1)$
& $\bbZ_2^4$
& $\Tindex{}{E}{7,2}{31}$
&\cite[Lem.~7, $\tilde n=1$]{F}
\\
\hline

$27$
&\ditto
& $5$
& $248$
& $(32,24,1)$
& $\bbZ_2^5$
& $\Tindex{}{E}{8,2}{66}$
&\cite[Lem.~7, $\tilde n=2$]{F}
\\
\hline

\end{tabular}
\medskip
\caption{Exceptional Lie tori and their invariants}
\label{tab:exceptional}
\end{table}
We do not provide here precise definitions of the exceptional Lie tori,
because to do so would take us rather far afield into the
fascinating world of nonassociative tori.  Instead,
in each case we have given  a reference for the
definition in the last column of the
table.\footnote{We only cite the reference that we find most convenient in our context.
Additional and sometimes earlier references can be found in the cited articles as well
as in Section 7 of the survey article \cite{AF}.}
If the Lie torus is the untwisted Lie torus with the indicated
root-grading type and nullity $n$ (see Example \ref{ex:untwisted}), we indicate
this simply with the word \emph{untwisted}.
In the case of the Lie tori numbered 15--24 (resp.~25--27), the Lie torus
is constructed using the Kantor construction from a structurable torus
(resp.~quasi-torus) that is defined in the indicated reference.
(See \cite[Thm.~5.6]{AY} and \cite[Thm.~3]{F}.)
Also,  for the Lie tori numbered 25--27, the quantity $\tilde n$ used in the last column
is the integer denoted by $n$ in \cite[Lemma~7]{F}.

For each exceptional Lie torus $\cL$, Columns
1, 3 and 4 of the table contain the isomorphism invariants  of $\cL$
(besides the nullity) that are described in Theorem \ref{thm:fourinv}, namely
the root-grading type of $\cL$, the centroid rank of $\cL$ and the root-space rank vector
of $\cL$ respectively.
Column  5 contains  the isotopy invariant described in Proposition
\ref{prop:isotopyinv}, namely the quotient external-grading group $\Lm/\Gm(\cL)$ (up to isomorphism)
of  $\cL$.\footnote{In this column and subsequently, we denote the direct
sum of $s$ copies of the group $\bbZ_\ell$ of integers mod $\ell$ by
$\bbZ_\ell^s$ for $\ell\ge 1$ and
$s\ge 0$. (If $s=0$ this direct sum is 0.)}

Finally,  Column 6 contains the index of $\cL$ (see Remark \ref{rem:Titsindex}).
These indices were
calculated using Tits' classification of indices \cite[Table II]{T}, Theorem \ref{thm:max},
and the entries in  Columns 1, 3 and 4, together with some special arguments in a few cases.
(See \cite[\S14.2]{ABP3} for some similar calculations.)

\myhead{The structure theorem}
In about the last 15 years, structure theorems (coordinatization theorems)
have been proved for centreless Lie tori  of each root-grading type.
This is work of
(in alphabetical order) Allison, Benkart, Berman, Faulkner,
Gao, Krylyuk, Neher and Yoshii in various combinations beginning with \cite{BGK}.
The reader can consult
Section 7 of the survey article \cite{AF} for precise references.

It turns out from these theorems that the only
centreless Lie tori that are not fgc are the Lie tori
$\spl_{r+1}(\cA)$ defined exactly in (A) using an associative torus $\cA$ that
is not fgc.

The following theorem summarizes the results of the structure theorems
for fgc centreless Lie tori.  There is some work needed to translate
the known results into our form,
but it  is not difficult to supply these arguments and we omit them.

\begin{theorem}
\label{thm:structure}
If $\base$ is algebraically closed,
every fgc centreless Lie torus is bi-isomorphic and hence
isotopic and isomorphic to  either  a classical Lie torus or an exceptional Lie torus.
\end{theorem}

\section{Invariants of classical Lie tori}
\label{sec:invclass}

In this section,   we calculate the invariants described in Theorem \ref{thm:fourinv}
and Proposition \ref{prop:isotopyinv} for  classical Lie tori.
For this, we first need to calculate the centroid in each  case.

\begin{proposition}
\label{prop:centroidLT}
Let $\cL$ be $\spl_{r+1}(\cA)$ as in (A),
$\spu_{2r+m}(\cA,-,D)$ as in (BC--B), $\ssp_{2r}(\cA,-)$ as in (C),
or $\orth_{2r}(\cA)$ as in (D).  Correspondingly let
$\cM$ be  $\Mat_{r+1}(\cA)$,
$\Mat_{2r+m}(\cA)$, $\Mat_{2r}(\cA)$  or $\Mat_{2r}(\cA)$,
in which case $\cL$ is a Lie subalgebra of $\cM$ under the commutator product.
Also correspondingly, let $\cZ$ be
$Z(\cA)$, $Z(\cA,-)$, $Z(\cA,-)$ or $\cA$, and   regard
$\cM$ as a Lie algebra over $\cZ$, where the action of $\cZ$ on $\cM$ is by
left multiplication on entries.  Then
$\cL$ is a $\cZ$-subalgebra of $\cM$ and the map
$\rho : \cZ  \to \Cd_\base(\cL)$ defined by $\rho(z)(x) = z  x$ is a $\Lm$-graded algebra isomorphism.
\end{proposition}

\begin{proof} This can be proved using Corollary 5.16 and Theorem 4.18 of \cite{BN},
although care must be taken in low rank.  Instead, we present an argument
using base ring extension and results about finite dimensional simple Lie algebras from
\cite[Chap. X]{J}.   We record this for the algebra $\cL =\spl_{r+1}(\cA)$ as in (A),
with the other cases being similar.

It is clear that  $\cL$ is a $\cZ$-subalgebra of $\cM$ and that
that $\rho$ is an injective $\Lm $-graded algebra homomorphism.  So it remains to show that $\rho$ is surjective.  Let $C = \Cd_\base(\cL)$ and use $\rho$ to regard $C$ as an algebra over $\cZ$.

Now $\cZ \simeq \base[\Omega]$ and $C \simeq \base[\Gamma]$ as $\Lm$-graded algebras,
where $\Omega$ and $\Gamma$ are subgroups of $\Lm$.  (The first statement is clear
and the second is \eqref{eq:centstructure}.)  Hence, $\Omega$ is a subgroup of $\Gamma$
and $C$ is a free $\cZ$-module of rank $[\Gamma:\Omega]$.  So, to show that
$\rho$ is surjective, it suffices to show that $\rank_\cZ(C) \le 1$.

Note that  $\cL$ is a free $\cZ$-module (for example since $C$ is a free
$\cZ$-module and $\cL$ is a free $C$-module by Proposition \ref{prop:CmodLT}(ii)).

Next, since $\cL$ is perfect, we have $C= \Cd_\cZ(\cL)$,
where $ \Cd_\cZ(\cL)
= \set{c\in \End_\cZ(\cL) \suchthat c[x, y]
= [c(x), y] = [x, c(y)] \text{ for } x,y\in \cL}$.  So
we have a natural $\tZ$-algebra homomorphism
\begin{equation}
\label{eq:Cextend}
\tZ\ot_\cZ C = \tZ\ot_\cZ \Cd_\cZ(\cL) \mapsto \Cd_\tZ( \tZ\ot_\cZ\cL),
\end{equation}
where $\tZ$ is the quotient field of $\cZ$.  We  claim that this map is injective.
Indeed, any element of $\tZ\ot_\cZ C$ is of the form $z^{-1}\ot c$, where $0\ne z\in \cZ$ and $c\in C$.
But if this element is in the kernel of the map  \eqref{eq:Cextend} then so is $1\ot c$.  So
$1\ot cx = 0$ for $x\in \cL$, which
implies that $cx = 0$ for $x\in \cL$, since $\cL$ is a free $\cZ$-module. Thus $c = 0$,
and we have proved the claim.
So it suffices to show that
$\dim_\tZ(\Cd_\tZ( \tZ\ot_\cZ\cL)) \le 1$, or in other words that
$\tZ\ot_\cZ\cL$ is central over $\tZ$.

Since $\cA$ is a free $\cZ$-module by  Lemma \ref{lem:freemodule},
$\cA$ embeds naturally in the $\tZ$-algebra $\tZ \ot_\cZ \cA$. Moreover,
since $\cA$ is an fgc domain, it is easily checked that $\tZ \ot_\cZ \cA$ is a finite dimensional
central division algebra over $\tZ$.  Also, by definition, $\cL =  [\cM,\cM]$, so
$\tZ \ot_\cZ \cL = \tZ \ot_\cZ [\cM,\cM] \simeq [\tZ \ot_\cZ\cM,\tZ \ot_\cZ\cM]$
as $\tZ$-algebras, where the last holds since $\tZ/\cZ$ is a flat extension. But
$\tZ \ot_\cZ\cM = \tZ \ot_\cZ\Mat_{r+1}(\cA) \simeq \Mat_{r+1}(\tZ \ot_\cZ \cA)$, so
\[\tZ \ot_\cZ \cL =[\Mat_{r+1}(\tZ \ot_\cZ \cA),\Mat_{r+1}(\tZ \ot_\cZ \cA)].\]
Thus by \cite[Thm. X.8]{J}, $\tZ \ot_\cZ \cL$ is a finite dimensional
central simple Lie algebra over $\tZ$.
\end{proof}

\myhead{Parameterization of classical Lie tori}
To  tabulate the invariants of classical Lie tori, we need to view each of the  Constructions (A), (BC--B), (C) and (D) as a construction from a list of parameters.
We now do this  using for the most part the tensor product decomposition of associative tori.

\smallskip\par{\bf (A)}: Suppose $\cL = \spl_{r+1}(\cA)$ is a special linear Lie torus as in
Construction \ref{con:class}(A).
Then, as noted in Section \myref{sec:structure}, we can assume that
\[\cA = \cA_1\otimes \dots \otimes \cA_k \otimes R_q,\]
where  $k\ge 0$, $q\ge 0$ and $\cA_i = \cQ(\zeta_i)$
with $\zeta_i$ a root of unity $\ne 1$ in $\base^\times$ for $i=1,\dots,k$.  So we can view $\cL$ as being constructed from the
parameters
\begin{equation*}
\label{eq:parA}
r\ge 1,\ k\ge 0,\  \zeta_1,\dots,\zeta_k \in \base^\times \text{ and } \ q\ge 0.
\end{equation*}
The only restrictions on the integral parameters $r$, $k$ and $q$ are those indicated, and the only restrictions
on the parameters $\zeta_1,\dots,\zeta_k$ are
\[2 \le \order{\zeta_i} < \infty \quad \text{ for } 1\le i \le k,\]
where  $\order{\zeta_i}$ denotes the order of $\zeta_i$ in the group $\base^\times$.

\smallskip\par{\bf (BC--B)}:  Suppose $\cL = \spu_{2r+m}(\cA,-,D)$  is a special unitary Lie torus  
as in Construction \ref{con:class}(BC--B).  Then as noted in Section \myref{sec:structure} we can assume that
\begin{equation*}\label{eq:BCnot1}
(\cA,-) =(\cA_1,-) \otimes \dots \otimes (\cA_k,-) \otimes (\cA_{k+1},-) \otimes (R_q,1),
\end{equation*}
where $k\ge 0$, $q\ge 0$,
\begin{equation}\label{eq:BCnot2}
(\cA_i,-) = (\cQ(-1),\natural) \quad \text{ for } i=1,\dots,k,
\end{equation}
and
\begin{equation}\label{eq:BCnot3}
(\cA_{k+1},-) = (\base,1),\ (R_1,\natural) \ \text{ or } (\cQ(-1),*)
\end{equation}
(as associative tori).
Corresponding to these 3 choices for $(\cA_{k+1},-)$
we set
\begin{equation*}
\label{eq:BCnot4}
p := 0,\  1 \ \text{ or } 2.
\end{equation*}
Note that the grading group for $(\cA,-)$ is
$L = L_1\oplus \dots \oplus L_{k+2}$, where $L_1,\dots,L_k = \bbZ^2$, $L_{k+1} = \bbZ^p$ and $L_{k+2} = \bbZ^q$.
Moreover the set $L_+$ consisting of the degrees of nonzero homogeneous elements in $\cA_+$   is then determined
by $k$, $p$ and $q$. (For example if $k=1$, $p=1$ and $q\ge 0$, we have
$L_+ = 2L + L_{k+2} + \set{0,\ep_{11}+\ep_{21}, \ep_{12}+\ep_{21},\ep_{11}+\ep_{12} +\ep_{21}}$, where
$\set{\ep_{11},\ep_{12}}$ is a $\bbZ$-basis for $L_1$, and $\set{\ep_{21}}$ is a $\bbZ$-basis for $L_2$.)
Recall that
$D = \diag(d_1,\dots,d_m)$ and that $\delta_i$ is the degree of $d_i$ in $L$.
We note that if
the elements $d_2,\dots,d_m$ are replaced by nonzero scalar multiples ($d_1=1$ is fixed), then $\cL$ is not changed up to isomorphism (in fact bi-isomorphism) \cite[Cor. 6.6.4]{AB}.
Hence, we can view $\cL$ as being constructed from the  parameters
\begin{equation}\label{eq:parBC}
r\ge 1,\ k\ge 0,\ p\in \set{0,1,2},\ q\ge 0,\ m\ge 1 \text{ and } \delta_1,\dots,\delta_m \in L_+.
\end{equation}
The restrictions on the integral parameters $r,k,p,q,m$ are those indicated as well as the additional restriction
\begin{equation}
\label{eq:resBC}
m \ge 5 \ \text{ if }\  (r,k,p) = (1,0,0)
\end{equation}
imposed in Construction \ref{con:class}(BC--B).
The restrictions on the $\delta_i$'s  in $L_+$ are that
\[\delta_1=0 \andd \delta_i +2L \ne \delta_j + 2L \text{ for } i\ne j.\]

\smallskip\par{\bf (C)}: Suppose $\cL = \ssp_{2r}(\cA,-)$ is a special symplectic Lie torus  as in Construction \ref{con:class}(C).
Then we can assume as in (BC--B) above that
\begin{equation}
\label{eq:parC1}
(\cA,-)=  (\cA_1,-) \otimes \dots \otimes (\cA_k,-) \otimes (\cA_{k+1},-) \otimes (R_q,1)
\end{equation}
where $k,q\ge 0$ and $(\cA_1,-),\dots,(\cA_{k+1},-)$ satisfy
\eqref{eq:BCnot2} and \eqref{eq:BCnot3}.
Again, we define
\begin{equation}
\label{eq:parC2}
p = 0,\ 1 \text{ or } 2
\end{equation}
corresponding to the choice of $(\cA_{k+1},-)$ in \eqref{eq:BCnot3}.
Then we can view $\cL$ as constructed from the integral parameters
\begin{equation}
\label{eq:parC3}
r\ge 1,\ k\ge 0,\ p\in \set{0,1,2} \text{ and } q\ge 0,
\end{equation}
subject to the indicated restrictions as well as the additional restriction
\begin{equation}
\label{eq:resC}
(k,p) \ne (0,0), (0,1), (1,0) \quad \text{ if } r=1 \text{ or } 2
\end{equation}
imposed in Construction \ref{con:class}(C).

\smallskip\par{\bf (D)}: Suppose finally that $\cL = \orth_{2r}(R_q)$ is an orthogonal Lie torus as in Construction \ref{con:class}(D), where $q\ge 0$.  Then $\cL$ is constructed from the integral parameters
\begin{equation*}
\label{eq:parD}
r \ge 4 \text{ and } q \ge0.
\end{equation*}

\myhead{The invariants of classical Lie tori}
We can now calculate the invariants of classical Lie tori. These will appear in Tables \ref{tab:classical1} and \ref{tab:classical2}, where we use the parameterizations described above.  In the tables and subsequently, we also use the following additional notation:
\begin{itemize}
\item For each of the four constructions, we define two additional positive integers $d$ and $s$
in the Construction column of Table \ref{tab:classical1}. In the definition of $d$ in (BC--B) and~(C),
$\lfloor \hphantom{a} \rfloor$ is the floor function, so that $d = 2^k$ if $p=0$ or $1$, and $d = 2^{k+1}$ if $p=2$.
\item In the last column of Table \ref{tab:classical1}, the symbol
\ $\widehat{\vphantom{c}}$\, above an entry of a vector indicates that the entry is to be omitted when $r=1$.
\item In Construction (BC--B),
$L/2L$ (resp.~$L/(2L+L_{k+2})$) is  a vector space of dimension $2k+p+q$ (resp.~$2k+p$)
over the field $\bbZ_2$ of integers modulo 2.  In Table~\ref{tab:classical2}, we let $a$ (resp.~$b$)
denote the dimension of the $\bbZ_2$-vector space generated by the cosets represented by
$\delta_1,\dots,\delta_m$ in $L/2L$ (resp.~$L/(2L+L_{k+2})$).
\end{itemize}

\begin{proposition}
\label{prop:classinv}
If $\cL$ is a classical Lie torus depending on parameters as described above, then the root-grading type, the nullity, the centroid rank and the root-space rank vector of
$\cL$ are listed in Table  \ref{tab:classical1}, and the quotient
external-grading group of $\cL$ is listed in Table
\ref{tab:classical2}.
\end{proposition}

\newcommand\T{\rule{0pt}{5ex}}  
\newcommand\B{\rule[-2ex]{0pt}{0pt}} 

\begin{table}[ht]
\renewcommand{\arraystretch}{1}
\setlength\doublerulesep{1pt}
\setlength\extrarowheight{4pt}
\relscale{.87}
\centering
\begin{tabular} 
{p{2.3cm} ||  p{2.7cm} | p{1.4cm} | p{2.1cm} | p{3.7cm} }

\hspace{0pt}\drop Construction
& Root-grading type of $\cL$
& Nullity of $\cL$
& \drop$\crk(\cL)$
& \drop $\rkv(\cL)$
\\
\hline\hline

$\text{(A)}$  \nl
with\nl
$d = \prod_{i=1}^k \order{\zeta_i}$,\nl
$s = (r+1)d$
&$\type{A}{r}$
& $2k+q$
& $s^2-1$
& $(d^2)$
\\
\hline

$\text{(BC--B)}  $\nl
with\nl
$d =  2^{k+\lfloor \frac p2 \rfloor}$,\nl
$s = (2r+m)d$
&$\type{B}{r}$ \nl
	\rtj{if $(k,p)=(0,0);$} \nl\nl
	$\type{BC}{r}$ \nl
	\rtj{if $(k,p)\ne(0,0)$\hphantom{;}}
&$2k+p+q$
&$\displaystyle \T \frac {s(s-(-1)^k)}2  $\nl
	\strut\rtj{if $p\ne 1$;} \nl \nl
	$s^2-1$ \nl
	\rtj{if $p=1$} \nl
&$(m,\widehat{1})$ \nl
	\rtj{if $(k,p) = (0,0)$;} \nl\nl
	$(md^2,\widehat{d^2},\displaystyle \frac {d(d-(-1)^k)}2)$ \nl
	\rtj{if $(k,p) \ne (0,0)$}\nl
	\rtj{and $p\ne 1$;} \nl\nl
	$(2md^2,\widehat{2d^2},d^2)$ \nl
	\rtj{if  $p= 1$}
\\
\hline

$\text{(C)}$\nl
with\nl
$d =  2^{k+\lfloor \frac p2 \rfloor}$,\nl
$s = 2rd$
&$\type{C}{r}$
& $2k+p+q$
&$\T \displaystyle \frac {s(s+(-1)^k)}2 $  \nl
	\strut\rtj{if $p\ne 1$;} \nl\nl
$s^2-1$ \nl
	\rtj{if $p=1$}
&$(\widehat{d^2},\displaystyle \frac {d(d+(-1)^k)}2)$ \nl
	\rtj{if $p\ne 1$;} \nl\nl
$(\widehat{2d^2},d^2)$ \nl
	\rtj{if $p=1$}
\\
\hline

$\text{(D)}$\nl
with\nl
$d = 1$,\nl
$s = 2r$
&$\type{D}{r}$
& $n$
& \T $\displaystyle \frac {s(s-1)}2 $
& $\displaystyle (d^2) = (1)$
\\
\hline
\end{tabular}
\bigskip
\caption{Invariants of classical Lie tori---Part 1}
\label{tab:classical1}
\end{table}

\begin{table}[ht]
\renewcommand{\arraystretch}{1.3}
\setlength\doublerulesep{1pt}
\setlength\extrarowheight{1pt}
\relscale{.87}
\centering
\begin{tabular} 
{m{2.5cm} || m{3.5cm}  }

Construction
&$\Lm/\Gamma(\cL)$
\\
\hline\hline

$\text{(A)}$
& $\B \bigoplus_{i=1}^k \bbZ_{\order{\zeta_i}}^2$
\\
\hline

$\text{(BC--B)}$
&$\bbZ_2^{2k+p+a-2b} \oplus \bbZ_4^b$
\\
\hline

$\text{(C)}$
& $\bbZ_2^{2k+p}$
\\
\hline

$\text{(D)}$
& $\set{0}$
\\
\hline
\end{tabular}
\bigskip
\caption{Invariants of classical Lie tori---Part 2}
\label{tab:classical2}
\end{table}

\begin{proof}  We outline the proof for special symplectic Lie tori.  The interested
reader will be able to supply the missing details in this case
and provide the arguments in the other three cases.   (Admittedly more work is involved for special unitary Lie tori, but the approach is the same.)

Let  $\cL = \ssp_{2r}(\cA,-)$  with the assumptions
and notation as in (C) above, and let
\[\cZ = Z(\cA,-).\]
Then, by Proposition \ref{prop:centroidLT}, $\Mat_{2r}(\cA)$ is a Lie algebra over $\cZ$  under the commutator product and $\cL$ is a $\cZ$-subalgebra of $\Mat_{2r}(\cA)$.

To compute some of the invariants of $\cL$, it will be helpful to work in a larger $\cZ$-subalgebra
$\cU$ of the Lie algebra $\Mat_{2r}(\cA)$.  Let
$G = \left[\begin{smallmatrix} 0&J_r  \\-J_r&0 \end{smallmatrix}\right] \in\Mat_{2r}(\base)$ and
\begin{align}
\notag
\cU &=
\set{X\in \Mat_{2r}(\cA) \suchthat G^{-1}\bar X ^t G = -X}\\
\label{eq:Cinv1}
&=\set{
\begin{bmatrix}
 A& B  \\C& -J_r \bar A^t J_r
\end{bmatrix}
\suchthat
A, B, C \in \Mat_r(\cA),\  J_r \bar B^t J_r = B,\
 J_r \bar C^t J_r = C}.
\end{align}
Then, as we saw in Construction \ref{con:class}(C),
\begin{equation}\label{eq:Cinv2}
\cL
= \set{X\in \cU \suchthat \ \trace(X) \in [\cA,\cA]}
\end{equation}
Also, it is known that $\cA = Z(\cA) \oplus [\cA,\cA]$ (see \cite[Lemma 5.1.3]{AB} for
a proof).  Hence
\begin{equation}\label{eq:Cinv3}
\cA_- = (Z(\cA)\cap \cA_-) \oplus ([\cA,\cA]\cap \cA_-).
\end{equation}
Thus
\begin{equation}\label{eq:Cinv4}
\cU = (Z(\cA)\cap \cA_-)  I_{2r} \oplus \cL
\end{equation}
Indeed the inclusion from right to left is clear, and the reverse inclusion
follows easily using \eqref{eq:Cinv1},  \eqref{eq:Cinv2} and \eqref{eq:Cinv3}.

Next let $\fh = \sum_{i=1}^{r} \base (e_{ii} - e_{2r+1-i,2r+1-i})$.  Recall from  Construction \ref{con:class}(C) that
$\fh$ is a split toral $\base$-subalgebra of $\cL$ with irreducible finite root system
$\De = \Delta_\base(\cL,\fh)$ of type $\type{C}{r}$.
In fact if we define $\ep_i\in \fh^*$ for $1\le i \le r$ by $\ep_i(e_{jj} - e_{2r+1-j,2r+1-j}) = \delta_{ij}$
we have
\begin{equation}
\label{eq:Cinv5}
\De = \set{\ep_i - \ep_j \mid 1\le i \ne j \le r} \cup \set{\pm(\ep_i + \ep_j) \mid 1\le i \le j \le r},
\end{equation}
with
\begin{equation}
\label{eq:Cinv5a}
\begin{aligned}
\cL_{\ep_i-\ep_j} &= \set{ a\, e_{ij} - \bar a\, e_{2r+1-j,2r+1-i} \suchthat a\in \cA}
	&&\text{for }1 \le i \ne j \le r,\\
\quad\cL_{\ep_i+\ep_j} &= \set{ a\, e_{i,2r+1-j} - \bar a\, e_{j,2r+1-i} \suchthat a\in \cA}
	&&\text{for } 1 \le i < j \le r,\\
\quad\cL_{-\ep_i-\ep_j} &= \set{ a\, e_{2r+1-i,j} - \bar a\, e_{2r+1-j,i} \suchthat a\in \cA}
	&&\text{for } 1 \le i < j \le r,\\
\cL_{2\ep_i} &= \set{ h\, e_{i,2r+1-i} \suchthat h\in \cA_+}
	&&\text{for } 1\le i  \le r.
\end{aligned}
\end{equation}
($\cL_0$ is the set of diagonal matrices in $\cL$.)
Recall also that $\cL$ is an fgc centreless Lie torus of type $(\De,\Lm)$ with gradings
determined by $\fh$ and $\cA$, where $\Lm$ is the grading group of $\cA$.  So,
as already observed in  Construction \ref{con:class}(C), the root-grading type of $\cL$
is $\type{C}{r}$.

Next  we have
\[
(\cA,-) =(\cA_1,-) \otimes \dots \otimes (\cA_{k+2},-),
\]
where for convenience we have set $(\cA_{k+2},-) = (R_q,1)$.
Hence, by definition, the grading group of $\cA$ is
\[\Lm = \Lm_1\oplus \dots \oplus \Lm_{k+2},\]
where $\Lm_i$ is the grading group of $\cA_i$ for $1\le i \le k+2$.
But $\Lm_1,\dots,\Lm_k = \bbZ^2$, $\Lm_{k+1} = \bbZ^p$ and $\Lm_{k+2} = \bbZ^q$,
so $\Lm \simeq \bbZ^{2k + p + q}$.  Hence the nullity of $\cL$ is $2k + p + q$.

Now by Proposition  \ref{prop:centroidLT}, we have a graded isomorphism
$\rho : \cZ \to \Cd_\base(\cL)$,
which we now use to identify
\begin{equation}
\label{eq:Cinv6}
\Cd_\base(\cL) =  \cZ.
\end{equation}
Hence, we have
\[\Gm := \Gm(\cL) = \Gm(\cA,-),\]
where,  as in Section \myref{sec:structure}, $\Gm(\cA,-) :=  \supp_\Lm(\cZ)$.

As we observed in Section \myref{sec:structure}, we have
\begin{equation}
\label{eq:Cinv7}
\cZ =\cZ_1 \otimes \dots \otimes \cZ_{k+2},
\end{equation}
where $\cZ_i = Z(\cA_i,-)$.  Hence
$\Gm = \Gm_1\oplus \dots \oplus \Gm_{k+2}$, where
$\Gm_i = \supp_{\Lm_i}(\cZ_i)$ for $1\le i \le k+2$.
One checks that $\Gm_i = 2\Lm_i$ for $1\le i \le k+1$, and clearly
$\Gm_{k+2} = \Lm_{k+2}$,
so
\[\Gm = 2\Lm_1\oplus \dots \oplus 2\Lm_{k+1} \oplus \Lm_{k+2},\]
Thus, $\Lm/\Gm(\cL) = \Lm/\Gm \simeq \bbZ_2^{2k+p}$.

Observe next
that by Lemma \ref{lem:freemodule}(ii),
$\cA$ is a free $\cZ$-module with
\begin{equation}\label{eq:rankZ0}
\rank_{\cZ}(\cA) = [\Lm : \Gm] = 2^{2k+p} =
\begin{cases}
	d^2 &\text{if $p\ne 1$,}\\
	2d^2 &\text{if $p=1$,}
\end{cases}
\end{equation}
where recall that  $d = 2^{k+\lfloor \frac p2 \rfloor}$.
Also, by Lemma \ref{lem:freemodule}(ii),
$\cZ(\cA) \cap \cA_-$, $\cA_+$ and $\cA_-$  are free $\cZ$-modules of  finite rank.
Moreover,
\begin{gather}\label{eq:rankZ1}
\rank_\cZ(\cZ(\cA) \cap \cA_-) = \delta_{1p},\\
\label{eq:rankZ2}
\rank_\cZ(\cA_+) =
\begin{cases}	
	\frac {d(d+(-1)^k)}2 &\text{if $p\ne 1$},\\
	d^2 &\text{if $p=1$},
\end{cases}\\
\label{eq:rankZ3}
\rank_\cZ(\cA_-) =
\begin{cases}	
	\frac {d(d-(-1)^k)}2 &\text{if $p\ne 1$},\\
	d^2 &\text{if $p=1$}.
\end{cases}
\end{gather}
We will justify these equalities at the end of the proof, but for the moment
we assume that they hold and use them to calculate the remaining
invariants.

First
\begin{align*}
\crk(\cL) &= \rank_{\Cd_\base(\cL)}(\cL)= \rank_\cZ(\cL) &&\text{by \eqref{eq:Cinv6}}\\
&= \rank_\cZ(\cM)- \rank_\cZ(Z(\cA)\cap \cA_-) &&\text{by \eqref{eq:Cinv4}}\\
&= \rank_\cZ(\cM)- \delta_{1p} &&\text{by \eqref{eq:rankZ1}}\\
&= r^2 \rank_\cZ(\cA) +
	2\big(\frac{r(r-1)}2 \rank_\cZ(\cA) + r\rank_\cZ(\cA_+)\big)
	- \delta_{1p} &&\text{by \eqref{eq:Cinv1}}\\
&= (2r^2-r) \rank_\cZ(\cA) +2r \rank_\cZ(\cA_+)	- \delta_{1p}.
\end{align*}
If we plug in the expressions  \eqref{eq:rankZ0} and  \eqref{eq:rankZ2}
for $\rank_\cZ(\cA)$ and $\rank_\cZ(\cA_+)$ into this last expression and simplify,
we obtain the values of $\crk(\cL)$ appearing in Table \ref{tab:classical1}.

Also,
\[
\rkv(\cL)
= (\widehat{\rank_{\cZ}(\cL_{\ep_1-\ep_2})},\rank_{\cZ}(\cL_{2\ep_1})) = (\widehat{\rank_\cZ(\cA)},\rank_\cZ(\cA_+))
\]
using \eqref{eq:Cinv5a}.
Again, plugging in our expressions for $\rank_\cZ(\cA)$ and $\rank_\cZ(\cA_+)$,
we obtain the values of $\rkv(\cL)$ appearing in Table \ref{tab:classical1}.

We conclude the proof by justifying \eqref{eq:rankZ1}, \eqref{eq:rankZ2} and \eqref{eq:rankZ3}.

Suppose first that $p\ne 1$.  Then $(\cA,-)$  is of first kind, so we have
\eqref{eq:rankZ1}.
Also, if $p=0$, we have
\[(\cA,-) = \underbrace{
(\cQ(-1),\natural) \otimes \dots \otimes (\cQ(-1),\natural)}_{\text{$k$ factors}}
 \otimes (R_q,1),\]
and  \eqref{eq:rankZ2} and \eqref{eq:rankZ3} and be proved simultaneously by
induction on $k$ using \eqref{eq:Cinv7}.  We leave the details to the reader.
Moreover, the equations \eqref{eq:rankZ2} and \eqref{eq:rankZ3} for the case
$p=2$ can  easily be deduced from the equations \eqref{eq:rankZ2} and \eqref{eq:rankZ3}
for the case $p=0$ (tensor with $(\cQ(-1),*)$).  Again we leave the details to the reader.

Finally, if $p=1$, then  $(\cA,-)$ is of second kind
and our equalities follow from \eqref{eq:Zi2}.
\end{proof}

\begin{remark} \label{rem:Titsclassical}
In this remark, we list the index of each classical Lie torus $\cL$ (see Remark \ref{rem:Titsindex}) using the notation of Table \ref{tab:classical1}.  We computed these indices using the methods outlined in  the proofs of Propositions \ref{prop:centroidLT} and   \ref{prop:classinv}, together with the detailed information
about the indices of classical algebraic groups found in \cite[Table II]{T}.  We omit any details
since, as we have mentioned, we do not use the index in this article.

If $\cL = \spl_{r+1}(\cA)$  is as in (A) above, then $\cL$ has
index $\Tindex{1}{A}{s-1,r}{(d)}$.
Next, if $\cL = \spu_{2r+m}(\cA,-,D) $
as in (BC--B), then $\cL$
has index
\begin{align*}
&\text{$\Tindex{}{C}{\frac s2,r}{(d)}$
	if $p\ne 1$ and $k$ is odd;}\\
&\text{$\Tindex{}{B}{\frac {s-1}2,r}{}$
	if $p \ne 1$, $(k,p) = (0,0)$ and $m$ is odd;}\\
&\text{$\Tindex{t}{D}{\frac s2,r}{(d)}$
	if $p \ne 1$, $k$ is even and either $(k,p) \ne (0,0)$ or $m$ is even; and}\\
&\text{$\Tindex{2}{A}{s-1,r}{(d)}$
	if $p=1$.}
\end{align*}
(In the second last case $t=1$ or $2$, and one can write down necessary and sufficient conditions involving the parameters \eqref{eq:parBC} for $t$ to be $1$.)
Further,  if $\cL = \ssp_{2r}(\cA,-)$ as in (C), then $\cL$
has index
\begin{align*}
&\text{$\Tindex{1}{D}{\frac s2,r}{(d)}$
	if $p\ne 1$ and $k$ is odd;}\\
&\text{$\Tindex{}{C}{\frac s2,r}{(d)}$
	if $p\ne 1$ and $k$ is even;
and}\\
&\text{$\Tindex{2}{A}{s-1,r}{(d)}$
	if $p=1$.}
\end{align*}
Finally, if $\cL = \orth_{2r}(\cA)$ as in (D), then
$\cL$ has index  $\Tindex{1}{D}{\frac s2,r}{(1)} = \Tindex{1}{D}{r,r}{(1)}$.
\end{remark}

\section{The isomorphism problem}
\label{sec:isomorphism}

To provide a classification of fgc centreless Lie tori up to isomorphism,
it remains to solve the isomorphism problem for fgc centreless Lie tori as they are
described in Theorem \ref{thm:structure}.

In  this section, we describe the results about the isomorphism problem that we can deduce using our isomorphism invariants and their values listed in Tables \ref{tab:exceptional} and \ref{tab:classical1}.

\begin{theorem}\
\label{thm:disjoint}
\begin{itemize}
\item [(i)]  The classes of classical Lie tori and exceptional Lie tori are disjoint.
That is, there is no Lie algebra that is isomorphic to a classical Lie torus and to an exceptional
Lie torus.
\item [(ii)] The classes of classical Lie tori obtained using constructions (A), (BC--B), (C) and (D)
are pairwise disjoint.
\end{itemize}
\end{theorem}

\begin{proof} (i): Suppose for contradiction that $\cL$ is a Lie algebra that is isomorphic
to a classical Lie torus and an exceptional Lie torus.  Then, comparing Tables \ref{tab:exceptional} and \ref{tab:classical1}, we see that $\cL$ has root-grading type  $\type{A}{1}$, $\type{A}{2}$ or
$\type{C}{3}$.  Hence, by Table \ref{tab:exceptional}, $\crk(\cL) = 78$ or $133$.  But from
Table  \ref{tab:classical1}, we see that $\crk(\cL)$ has the form
$s^2-1$ or $\frac {s(s\pm 1)}2$ for some positive integer $s$.  This rules
out $\crk(\cL)=133$, so $\crk(\cL) = 78$.  Hence by
Table \ref{tab:exceptional}, $\cL$ has root grading type $\type{A}{2}$.
Thus, by Table  \ref{tab:classical1}, $\crk(\cL) = s^2-1$ for some positive integer $s$, so
$78 = s^2-1$.
This is a contradiction.\footnote{(i) can also be seen by comparing indices or, in
all but one case, absolute types (see for example \cite[\S 3.3]{ABP3} for this terminology).}

(ii):  Suppose for contradiction that there is a Lie algebra that is isomorphic
to Lie tori $\cL$ and $\cL'$ coming from two different constructions from the list
(A), (BC--B), (C) and (D).
We will use the notation of Section \ref{sec:invclass} for $\cL$
and corresponding primed notation for $\cL'$.
Since the root-grading type is an isomorphism invariant, we see from
Table  \ref{tab:classical1} that we must have one of the following (up to an exchange of $\cL$ and $\cL'$):
\begin{itemize}
\item[(a)] $\cL$ arises from  (A) with $r=1$, and $\cL'$ arises from (BC--B) with $(r',k',p') = (1,0,0)$;
\item[(b)] $\cL$ arises from  (A) with $r=1$, and $\cL'$ arises from  (C) with $r'=1$;  or
\item[(c)] $\cL$ arises from  (BC--B) with $(r,k,p) = (2,0,0)$, and $\cL'$ arises from  $(C)$ with $r'=2$.
\end{itemize}

Suppose first that (a) holds.  Then we have
$r=1$ and $s=2d$ in construction (A);  and we have $d'=1$  and $s' = m'+2$ in
construction (BC--B).  Comparing the root-space rank vectors in Table   \ref{tab:classical1}
we see that $d^2 = m'$, whereas comparing centroid-rank vectors we see that
$4d^2-1 = \frac {s'(s'-1)}2 = \frac {(m'+2)(m'+1)}2$.  So
$4m' -1 =  \frac {(m'+2)(m'+1)}2$ which forces $m'=1$ or $4$.
But these values of $m'$ were excluded in~\eqref{eq:resBC}.

Suppose next that (b) holds.
Then arguments (which we leave to the reader) that are similar to the one in (a) handle all but one case:
\[\text{$\cL = \spl_2(\cA)$ and $\cL' =  \ssp_2(\cA',-)$ with $p' = 1$.}\]
In this case, we have
$k'\ne 0$ by \eqref{eq:resC}, so $d' \ge 2$.
We now indicate,  omitting the details, how this leads to a contradiction using
base-ring extension and a theorem about finite dimensional central simple Lie algebras from \cite{J}.
(Alternatively, one could compare the indices of both sides using  Remark
\ref{rem:Titsclassical}.)   Let
$\cZ = Z(\cA)$ and $\cZ' = Z(\cA',-)$,
with quotient fields $\tZ$ and $\tZp$ respectively.
Now, as in the proof of Proposition \ref{prop:classinv}, we see that
$\cA$ has rank $d^2$ over $\cZ$ and $\cA'$ has rank $2{d'^2}$ over $\cZ'$.
Moreover, as in the proof of Proposition \ref{prop:centroidLT}, we see that
that $\tZ \otimes_{\cZ} \cA$ is a division
algebra of dimension $d^2$ over its centre $\tZ$, and $(\tZp \otimes_{\cZ'} \cA,-)$ is a division algebra with involution of dimension $2{d'}^2$ over its centre  $\tZp$ (as an algebra with involution).
Also, again as in the proof of Proposition \ref{prop:centroidLT}, we see that the central closures
of $\cL$ and $\cL'$ are respectively $\spl_2(\tZ \otimes_\cZ \cA)$ and
$\ssp_2(\tZp \otimes_{\cZ'} \cA,-)$.  (These last two Lie algebras are defined
as in Constructions \ref{con:class}(A) and (C).) So, by Remark \ref{rem:closureisom}, we have
$\spl_2(\tZ \otimes_{\cZ} \cA) \simeq \ssp_2(\tZp \otimes_{\cZ'} \cA,-)$.
Since $d' \ge 2$, this contradicts the last
statement in Theorem X.11 of \cite{J}.

Finally,  case (c) is handled easily using the method in (a)
and the   exclusion~\eqref{eq:resC}.\end{proof}

Combining Theorems \ref{thm:structure} and \ref{thm:disjoint},  we see that the isomorphism problem for fgc centreless Lie tori  reduces to 5 separate problems, one for exceptional Lie tori and one for each of the four Constructions (A), (BC--B), (C) and (D) of classical Lie tori.

The next theorem solves the isomorphism problem for Constructions (C) and~(D).

\begin{theorem}\
\label{thm:isomorphismCD}
\begin{itemize}
\item [(i)]  Let
$P_C$ be the set of all vectors $(r,k,p,q)$ in $\bbZ^4$ such that
$r\ge 1$, $k\ge 0$, $p\in \set{0,1,2}$, $q\ge 0$ and if, $r=1$ or $2$,
$(k,p) \notin \set{(0,0), (0,1), (1,0)}$.  Then the map that sends $(r,k,p,q)$ to the isomorphism class represented by $\ssp_{2r}(\cA,-)$,
where $(\cA,-)$ is the tensor product of basic associative tori with involution constructed from $(k,p,q)$ as in  \eqref{eq:parC1}--\eqref{eq:parC3}, is a bijection from $P_C$ onto the set of isomorphism classes of special symplectic Lie tori.
Moreover, special symplectic tori are classified by their root-grading type, nullity and root-space rank vector.
That is, two special symplectic Lie tori are isomorphic if and only if they have same root-grading type, the same nullity
and the same root-space rank vector.
\item [(ii)]  Let
$P_D = \set{(r,n)\in \bbZ^2 \suchthat r\ge 4,\ n\ge 0}$.
Then the map that sends $(r,n)$ to the isomorphism class represented by $\orth_{2r}(R_n)$ is a bijection
from $P_D$ onto the set of isomorphism classes of orthogonal Lie tori.
Moreover, orthogonal Lie tori are classified by their root-grading type and nullity.
\end{itemize}
\end{theorem}

\begin{proof}
(i): Before beginning, let $\bbN = \set{k\in \bbZ \suchthat k\ge 1}$, $S = \bbN\times \set{0,1,2}$, and
define $f : S \to \bbN$ by
\[f(k,p) =  \begin{cases}	
	 2^{k+\frac p2 -1}(2^{k+\frac p2} + (-1)^k) &\text{if $p\ne 1$},\\
	 2^{2k} &\text{if $p=1$}.
\end{cases}\\ \]
It is not difficult to show that if $(k,p)\in S$ then
\begin{equation}
\label{eq:fprop1}
f(k,p) = 1 \iff (k,p)\in \set{(0,0), (1,0), (0,1)};
\end{equation}
and that
\begin{equation}
\label{eq:fprop2}
f|_{S\setminus\set{(0,0), (1,0), (0,1)}} \text{ is one-to-one}.
\end{equation}
We leave these facts for the reader to check.

Now to begin the  proof of (i), observe that the map described in the first statement of (i) is surjective by the discussion of parameterization in Section \ref{sec:invclass}.

Next suppose that $(r,k,p,q)$ and $(r',k',p',q')$ are in $P_C$.
We let   $\cL = \ssp_{2r}(\cA,-)$,
where $(\cA,-)$ is the tensor product of basic associative tori with involution constructed from $(k,p,q)$
as in  \eqref{eq:parC1}--\eqref{eq:parC3}, and we let
$\cL' = \ssp_{2r'}(\cA',-)$, where $(\cA',-)$  is obtained in the same way from $(k',p',q')$.
Observe that by Table \ref{tab:classical1}, we have
\[\rkv(\cL) = (\widehat{2^{2k+p}},f(k,p)),\]
and we have a similar expression for $\rkv(\cL')$.

We will show that the following statements are equivalent:
\begin{itemize}
  \item[(a)] $\cL \simeq \cL'$,
  \item[(b)]  $\cL$ and $\cL'$ have the same root-grading type, the same nullity and the same root-space rank vector,
  \item[(c)]  $(r,k,p,q)=(r',k',p',q')$.
\end{itemize}
Note that this will compete the proof of both of the statements in (i).

Now ``(a) $\Rightarrow$ (b)'' holds by Theorem \ref{thm:fourinv}, and ``(c) $\Rightarrow$ (a)'' is trivial.  Thus it suffices
to show that ``(b) $\Rightarrow$ (c)''.  So, suppose that (b) holds.  Then
$r=r'$,
\begin{gather}
\label{eq:Cclass1}
2k+p+1 = 2k' + p' + q',\\
\label{eq:Cclass2}
f(k,p) = f(k',p'), \text{ and }\\
\label{eq:Cclass3}
2^{2k+p} = 2^{2k'+p'} \text{ if } r\ge 2.
\end{gather}

Suppose first that $(k,p)\in \set{(0,0), (1,0), (0,1)}$.
Then, $f(k,p) = 1$ by \eqref{eq:fprop1}, so $(k',p')\in \set{(0,0), (1,0), (0,1)}$
by \eqref{eq:fprop1} and \eqref{eq:Cclass2}.  But $r\ge 3$ by definition of $P_C$,
so by \eqref{eq:Cclass3} we have $2k+p = 2k'+p'$.  Hence $(k,p) = (k',p')$.
Finally, by \eqref{eq:Cclass1}, we have $q = q'$.

Lastly, suppose that $(k,p)\notin \set{(0,0), (1,0), (0,1)}$, so by the argument just given
$(k',p')\notin \set{(0,0), (1,0), (0,1)}$.  Thus, by
\eqref{eq:Cclass2} and \eqref{eq:fprop2}, we have $(k,p) = (k',p')$, and therefore also
$q = q'$ as above.

(ii): Since $\orth_{2r}(R_n) \simeq \orth_{2r}(\base)\otimes R_n$, this follows from well-known facts about  multiloop algebras (see for example \cite[Cor.~8.19]{ABP3}).  From our point of view here, it also follows immediately
using the argument in (i) and Table \ref{tab:classical1}.
\end{proof}

\begin{corollary}  Fix $r\ge 3$.  Then the fgc centreless Lie tori of type $\type{C}{r}$
are classified by their nullity and root-space rank vector.
\end{corollary}

\begin{proof}  Suppose that $\cL$ and $\cL'$ are fgc centreless Lie tori of type $\type{C}{r}$ with the same  nullity and the same root-space rank vector.  If $\cL$ and $\cL'$ are classical, we have $\cL \simeq \cL'$
by Theorem \ref{thm:isomorphismCD}(i).  On the other hand if $\cL$ and $\cL'$ are exceptional, then $r=3$ and $\cL \simeq \cL'$ by Table \ref{tab:exceptional}.  Finally, if $\cL$ is exceptional and $\cL'$ is classical, then $r=3$
and $\rkv(\cL') = \rkv(\cL) = (8,1)$  by Table \ref{tab:exceptional}, which contradicts Table~\ref{tab:classical1}.
\end{proof}

\myhead{The remaining isomorphism problems}
The 3 remaining isomorphism problems are now listed, together with some comments.  We will say more about each of these problems in the next section.

(1)  \emph{The isomorphism problem for exceptional Lie tori.} It follows looking at root-space rank vectors in Table \ref{tab:exceptional} that
the only possible isomorphisms between  Lie tori
of a given root-grading type and nullity are between
the tori numbered 20, 22 and 23, or between the tori
numbered 21 and 24.   \emph{So it remains to decide if any such
isomorphisms exist}.  Note however that the tori numbered 20, 22 and 23
have distinct quotient external-grading groups, as do
the tori numbered 21 and 24.  Therefore, the exceptional Lie tori
listed in Table \ref{tab:exceptional} are pairwise not isotopic.

(2) \emph{The isomorphism problem for special linear Lie tori.}
 The Lie tori in construction (A) are not classified
by the four  isomorphism invariants from Theorem \ref{thm:fourinv}, even in
nullity 2.\footnote{For example, it follows from \cite[Thm.~11.3.2]{ABP3} that if $\zeta$
is a 5th root of unity, then $\spl_{r+1}(\cQ(\zeta))$ and $\spl_{r+1}(\cQ(\zeta^2))$  are not isomorphic.  However,
by Table \ref{tab:classical1}, they each have nullity 2, root-grading type $\type{A}{r}$,  centroid rank $25(r+1)^2-1$ and root-space rank vector (25).}
(The index adds no extra information;
that is, if the four invariants
match for two special linear Lie tori, one can check that the
indices match as well.)
However, it is clear that two special linear Lie tori that have the same root-grading type, the same nullity and the same root-grading
rank vector,
have a common $r$, a common nullity for their coordinate associative tori $\cA$ and $\cA'$, and a common
rank for $\cA$ and $\cA'$  as modules over their centres.  So these quantities can be
fixed  when considering the problem.

(3) \emph{The isomorphism problem for special unitary Lie tori.}
The Lie tori in construction (BC--B) are not classified
by the four  isomorphism invariants from Theorem \ref{thm:fourinv}, even in
nullity 3.  (Again, including the index does not provide enough information for classification.)
However, one can check (using the argument in the proof of Theorem \ref{thm:isomorphismCD}(i))) that with one exception,
two special unitary Lie tori that have the same root-grading type, nullity, centroid rank and root-grading
rank vector have a common $r$, a common coordinate associative torus with involution $(\cA,-)$
up to isomorphism, and a common value for $m = |D|$.  So these entities
can be fixed when considering the problem.

In summary, the classification of fgc centreless Lie tori up  to isomorphism is
reduced to the  separate
isomorphism problems for (1) five particular exceptional Lie tori, (2)  special linear Lie tori,  and (3) special unitary Lie   tori.

\section{Conjugacy and its implications}
\label{sec:conjimp}
In this  section  we consider the three isomorphism
problems just discussed under a conjugacy assumption for certain maximal split toral
$\base$-subalgebras.

\myhead{A conjugacy assumption}
We say that an fgc centreless Lie algebra $\cL$
satisfies \emph{Assumption (C)} if the following holds:
\begin{itemize}
\item[(C)]
If $\cL$  has the  graded structure
$\cL = \bigoplus_{(\al,\lm)\in Q(\Delta)\times\Lm}\dL{\al}{\lm}$ of
a Lie torus of type $(\De,\Lm)$ with $\fh = \cL_0^0$
and if (the same) $\cL$ has
the graded structure $\cL = \bigoplus_{(\al',\lm')\in Q(\Delta')\times\Lm'}{\cL'}_{\al'}^{\lm'}$  of
a Lie torus of type $(\De',\Lm')$ with $\fh' = {\cL'}_0^0$,
then there exists $\ph\in \Aut(\cL)$ such that $\ph(\fh) = \fh'$.
\end{itemize}

Less precisely, this assumption says that  two maximal split toral  $\base$-subalge\-bras
of $\cL$ that arise from Lie torus structures on $\cL$ are conjugate under the action of $\Aut(\cL)$.

\begin{remark}
\label{rem:CGP}  Our
motivation for making Assumption (C) in the theorems below is
work in progress by V.~Cher\-nou\-sov, P.~Gille and A.~Pianzola \cite{CGP}.  This work will show that
Assumption (C) holds for \emph{any} fgc centreless Lie torus and therefore the assumption
is superfluous.
It is already known that this is the case for untwisted Lie tori  \cite{P1}.
\end{remark}

An immediate consequence of Assumption (C) together with Theorem  \ref{thm:isotopychar}
is the following:

\begin{theorem}
\label{thm:isis}  Suppose that
$\cL$  and $\cL'$ are fgc centreless Lie tori satisfying Assumption (C). Then
$\cL$ is isomorphic to $\cL'$ if and only if $\cL$ is isotopic to $\cL'$ .
\end{theorem}
\begin{proof}  Suppose that $\ph : \cL \to \cL'$ is an isomorphism.
By Assumption (C) we can assume that $\ph(\fh) = \fh'$, where
$\fh = \cL_0^0$ and $\fh' = \cLp_0^0$.  Then
$\ph$ is an isotopy by Theorem \ref{thm:isotopychar}.
\end{proof}

Putting  this result together with Proposition \ref{prop:isotopyinv}, we have:

\begin{corollary}
\label{cor:quotientinv} The quotient external-grading group is an isomorphism
invariant for an fgc centreless Lie torus satisfying Assumption  (C).
\end{corollary}

In the remaining sections,  we discuss the implications of Assumption (C) and our results  for the three  isomorphism problems listed at the end of Section \ref{sec:isomorphism}.

\myhead{Isomorphism of exceptional Lie tori}
In  the discussion of Problem 1 at the end of Section \ref{sec:isomorphism}, we used  the quotient external-grading group to show that the exceptional Lie tori in Table \ref{tab:exceptional} are pairwise not isotopic.
Therefore, it follows from Theorem \ref{thm:isis} that if exceptional
Lie tori satisfy Assumption (C), then  they are listed up to isomorphism
\emph{without redundancy} in Table  \ref{tab:exceptional}.

\myhead{Isomorphism of special linear Lie tori}
For special linear Lie tori, the quotient external-grading group $\Lm/\Gm$ adds no further information. That is, if the four invariants in Theorem
\ref{thm:fourinv} match for two special linear Lie tori, one can check that
the quotient external-grading groups also match.
However,  under assumption (C) we can prove the following theorem, which was proved in nullity 2
in \cite[Thm.~11.3.2]{ABP3} without Assumption (C).

\begin{theorem}
\label{thm:isomophismA}
Suppose that $\cA$ and $\cA'$ are fgc associative tori of nullity $n$, $r\ge 1$, and
the Lie tori $\spl_{r+1}(\cA)$ and $\spl_{r+1}(\cA')$ satisfy Assumption (C).
Then $\spl_{r+1}(\cA)$ and $\spl_{r+1}(\cA')$ are isomorphic if and only $\cA$ and $\cA'$
are isomorphic.
\end{theorem}

\begin{proof}  The implication from right to left is clear and so we consider only the converse.
This can be seen to follow using Assumption (C) from Theorems 8.6(ii), 9.11 and 10.6 of \cite{AF} in the
cases $r=1$, $r=2$ and $r\ge 3$ respectively.  However for the readers convenience we give a uniform argument
that follows the approach in Section 9 of \cite{AF}.

We begin arguing that (as noted in \cite[Remark 9.12]{AF}) $\cA\simeq\cA^{\text{op}}$, where
$\cA^{\text{op}}$ denotes the \emph{opposite} algebra of $\cA$ with product
$(x,y) \mapsto yx$.  Indeed, using the tensor product decomposition \eqref{eq:tensor} of $\cA$,
it is sufficient to consider the case when $\cA = \cQ(\zeta)$, where $\zeta$ is a root of unity.
But in this case we have $\cA\simeq\cA^{\text{op}}$ under the homomorphism
exchanging $\qg_1$ and $\qg_2$.

Let $\cL = \spl_{r+1}(\cA)$, $\cL' = \spl_{r+1}(\cA')$  and $\cL'' = \spl_{r+1}({\cA'}^{\text{op}})$.
Then  $\fh = \cL_0^0$, $\fh' = {\cL'}_0^0$ and $\fh'' = {\cL''}_0^0$ are identified
in Construction \ref{con:class}(A) with $\sum_{i=1}^{r} \base (e_{ii} - e_{2r+1-i,2r+1-i})$
in $\cL$, $\cL'$ and $\cL''$ respectively.  So we can identify $\fh$, $\fh'$ and $\fh''$
in the evident fashion.  In this way, $\fh = \sum_{i=1}^{r} \base (e_{ii} - e_{2r+1-i,2r+1-i})$ is a subalgebra of $\cL$, $\cL'$ and $\cL''$; and we have
\[\De := \De_\base(\cL,\fh) =  \De_\base(\cL',\fh) =  \De_\base(\cL'',\fh).\]

Assume now that $\spl_{r+1}(\cA) \simeq \spl_{r+1}(\cA')$.   Then,
by Assumption (C) applied to $\cL'$, we have an isomorphism $\ph : \cL \to \cL'$ such that $\ph(\fh) = \fh$.
Thus, as in the proof of Theorem \ref{thm:isotopychar}, $\ph$ induces a linear automorphism  $\hat\ph$ of $\fh^*$ such that
$\ph(\cL_\al) = \cL'_{\hat\ph(\al)}$ for $\al\in \fh^*$.  So $\hat\ph$ is an automorphism
of the root system $\De$.

Next define $\psi : \cL'\to \cL''$ by $\psi(x) =  -x^t$ for $x\in \cL'$.  Then
$\psi$ is an algebra isomorphism with $\psi(\fh) = \fh$, and we have
$\hat\psi = -1$.
Thus, replacing $\cA'$ by ${\cA'}^{\text{op}}$ and $\ph$ by $\psi\ph$ if necessary,
we can assume that $\hat\ph$ is in the Weyl group of $\De$.

So, replacing
$\ph$ by $\pi\ph$, where $\pi$ is conjugation by an appropriate permutation matrix over $\base$, we can assume that $\hat\ph = 1$.
Therefore, for $1\le i \ne j \le r+1$, we have a linear bijection $\ph_{ij} : \cA \to \cA'$ with
\[\ph(ae_{ij}) = \ph_{ij}(a) e_{ij}\]
for $a\in \cA$.

Now, if $a,b,c\in \cA$, we have
\[[[ae_{12},be_{21}],ce_{12}] = [abe_{11}-bae_{22},ce_{12}] = (abc + cba)e_{12}.\]
Applying $\ph$ to this equation, yields
\begin{equation}\label{eq:ph1}
\ph_{12}(a)  \ph_{21}(b) \ph_{12}(c) + \ph_{12}(c) \ph_{21}(b) \ph_{12}(a) =  \ph_{12}(abc+cba).
\end{equation}
If we put $a=c=1$ in this equation we get
\begin{equation}\label{eq:ph2}
\ph_{12}(1)  \ph_{21}(b) \ph_{12}(1) = \ph_{12}(b)
\end{equation}
for $b\in\cA$.
So
$\ph_{12} = \ell_{\ph_{12}(1)} r_{\ph_{12}(1)} \ph_{21}$, where
$\ell_{\ph_{12}(1)}$ and $r_{\ph_{12}(1)}$ in $\End_\base(\cA')$ are left and right multiplication respectively by $\ph_{12}(1)$. Thus, $\ell_{\ph_{12}(1)} r_{\ph_{12}(1)} = r_{\ph_{12}(1)} \ell_{\ph_{12}(1)}$
is invertible, and so  $\ph_{12}(1)$ is a unit in $\cA'$.

Replacing $\ph$ by $\mu\ph$, where $\mu$ is conjugation by $\diag(\ph_{12}(1)^{-1},1,\dots,1)$, we can
assume that $\ph_{12}(1) = 1$.  Hence, by \eqref{eq:ph2}, $\ph_{21} = \ph_{12}$. So putting $b=1$
in \eqref{eq:ph1}, we have $\ph_{12}(a) \ph_{12}(c) + \ph_{12}(c) \ph_{12}(a) = \ph_{12}(ac + ca)$ for $a,c\in \cA$.
That is $\ph_{12} : \cA \to \cA'$ is a \emph{Jordan isomorphism}.  Hence, since
$\cA'$ is a prime ring, a theorem of Herstein \cite[Thm.~3.1]{H} tells us that $\ph_{12}$ is either an isomorphism or an anti-isomorphism of $\cA$ onto $\cA'$.  Since $\cA'\simeq{\cA'}^{\text{op}}$ it follows that
$\cA \simeq \cA'$.
\end{proof}

If all special linear Lie tori satisfy Assumption (C), Theorem \ref{thm:isomophismA} reduces their
classification to Neeb's classification of fgc associative tori mentioned in Section
\myref{sec:structure}.

\myhead{Isomorphism of special unitary Lie tori}
If all special unitary Lie tori satisfy Assumption (C), Theorem \ref{thm:isis} tells us that their classification up to isomorphism is reduced to determining when two such Lie tori are isotopic.
We are optimistic that the latter can be accomplished along the lines of \cite[\S 7]{AB}, perhaps using a notion of isotope for graded hermitian forms  (following the philosopy of \cite{AF}).
However, at this point the isotopy problem for special unitary Lie tori is open.

\myhead{Summary} If all fgc centreless Lie tori satisfy Assumption (C),
our work in this article has reduced their classification  up to isomorphism to solving the isotopy problem for special unitary Lie tori.


\end{document}